\documentclass[
12pt,
paper=a4,
headings=small,
parskip=half,
]{scrartcl} 
\setkomafont{sectioning}{\rmfamily\bfseries} 
\usepackage[numbers]{natbib}

\usepackage{microtype}
\usepackage[latin1]{inputenc}
\usepackage{amsfonts}
\usepackage{mathdots}
\usepackage{amsmath}
\usepackage{amssymb}
\usepackage{nicefrac}
\usepackage{amsthm}

\usepackage{url}
\usepackage{hyperref}

\usepackage[matrix,arrow]{xy}
\usepackage{booktabs}

\newcommand{\IQ}{\mathbb{Q}}
		
\newcommand{\IZ}{\mathbb{Z}}
\newcommand{\IC}{\mathbb{C}}
\newcommand{\IH}{\mathbb{H}}
\newcommand{\IR}{\mathbb{R}}

\newcommand{\IF}{\mathbb{F}}

\newcommand{\IA}{\mathbb{A}}
\newcommand{\Ione}{\mathbf{1}}

\DeclareMathOperator{\St}{St}

\DeclareMathOperator{\cInd}{cInd}

\DeclareMathOperator{\diag}{diag}

\DeclareMathOperator{\nr}{Nr}
\DeclareMathOperator{\tr}{tr}

\DeclareMathOperator{\Mat}{Mat}
\DeclareMathOperator{\Sym}{Sym}

\DeclareMathOperator{\simi}{sim}
\DeclareMathOperator{\lcm}{lcm}
\DeclareMathOperator{\Groth}{\mathbf{R}}
\DeclareMathOperator{\Rep}{\mathbf{Rep}}
\DeclareMathOperator{\Irr}{\mathbf{Irr}}

\newtheorem{Theorem}{Theorem}[section]
\newtheorem{Lemma}[Theorem]{Lemma}
\newtheorem{Prop}[Theorem]{Proposition}
\newtheorem{Corollary}[Theorem]{Corollary}

\theoremstyle{definition}
\newtheorem{Def}[Theorem]{Definition}
\newtheorem{Not}[Theorem]{Notation}

\begin{document}

\author{
\begin{normalsize}
Mirko R\"osner
\end{normalsize}
}
\date{}

\title{
\begin{large}
Invariant Vectors for Weak Endoscopic and Saito-Kurokawa Lifts to GSp(4)
\end{large}
}

\maketitle

\begin{abstract}
Let $\IA$ be the ring of adele over a totally real number field $F/\IQ$. For cohomological cuspidal automorphic irreducible representations of $GSp(4,\IA)$ coming from weak endoscopic or Saito-Kurokawa Lifts, we determine the local invariant spaces under the first principal congruence subgroup at the non-archimedean places of $F$.
For $F=\IQ$ this gives rise to dimension formulas regarding certain subspaces of the inner cohomology of the genus two Shimura variety corresponding to the principal congruence subgroup of level $N=2$. We prove the conjectures made by Bergstr\"om, Faber and van der Geer in a recent paper.
\end{abstract}
\section*{Introduction}
Fix a totally real number field $F$ with adele ring $\IA=\prod'_v F_v$.
Suppose $\sigma$ is a generic irreducible admissible representation of $GL(2,F_v)$ at a non-archimedean place $v$. By a result of Casselman there is a unique ideal $\mathfrak{p}_v^n\subseteq\mathfrak{o}_v$, called the level of $\sigma$, such that the space of invariants in $\sigma_v$ under the congruence subgroup 
\begin{align*}K_{0}(\mathfrak{p}_v^n)=\{g\in GL(2,\mathfrak{o}_v)\,|\,g\equiv\left(\begin{smallmatrix}\ast&\ast\\&\ast\end{smallmatrix}\right)\mod\mathfrak{p}_v^n\}\end{align*}
is one-dimensional \citep{Casselmann_res_atk_leh}. B.~Roberts and R.~Schmidt have determined the corresponding invariant spaces for paramodular subgroups of $GSp(4,F_v)$ \citep{Roberts-Schmidt}. In this article we discuss results obtained in \citep{my_thesis} concerning invariant spaces of $GSp(4,F_v)$-representations under the first principal congruence subgroup
\begin{align*}
K(\mathfrak{p}_v)=\ker(GSp(4,\mathfrak{o}_v)\to GSp(4,\mathfrak{o}_v/\mathfrak{p}_v))
\text{.}
\end{align*}
For a cuspidal automorphic irreducible representation $\pi=\pi_\infty\pi_\mathrm{fin}$ of $GSp(4,\IA)$ we assume there are cuspidal automorphic irreducible representations $\sigma_1,\sigma_2$ of $GL(2,\IA)$ with equal central character such that for almost every place $v$
\begin{align*}L(\pi_v,s)=L(\sigma_{1,v},s)\cdot L(\sigma_{2,v},s)\text{,}
\end{align*}
then we call $\pi$ a weak endoscopic lift attached to $\sigma=(\sigma_1,\sigma_2)$, compare \citep[\S5.2]{Weissauer200903}.
We also consider the case where $\pi$ arises via the Saito-Kurokawa lift in the sense of \citep{PS-SaitoK} from some cuspidal automorphic irreducible representation $\sigma$ of $PGL(2,\IA)$.
For both lifts we address the question: \emph{How can the local invariant-space $\pi_{v}^{K(\mathfrak{p}_v)}$ be described as a representation of $GSp(4,\mathfrak{o}_v/\mathfrak{p}_v)$}?
%
This means we study the image of $\pi$ under the parahoric restriction functor of taking $K(\mathfrak{p}_v)$-invariants:
\begin{align*}
 \mathcal{F}_{\mathfrak{p}_v}:\Rep(GSp(4,F_v))\longrightarrow& \Rep(GSp(4,\mathfrak{o}_v/\mathfrak{p}_v))\\
(\pi,V_\pi)\longmapsto& (\pi|_{GSp(4,\mathfrak{o}_v)},V_\pi^{K(\mathfrak{p}_v)})\text{.}
\end{align*}
To an admissible complex linear representation of $GSp(4,F_v)$ it assigns the space of $K(\mathfrak{p}_v)$-invariants. Since $K(\mathfrak{p}_v)\trianglelefteq GSp(4,\mathfrak{o}_v)$ is a normal subgroup, this invariant space is preserved under the action of $GSp(4,\mathfrak{o}_v)$ and defines a representation of $GSp(4,\mathfrak{o}_v/\mathfrak{p}_v)$. The functor $\mathcal{F}_{\mathfrak{p}_v}$ is additive and exact because $K(\mathfrak{p}_v)$ is compact.
For non-cuspidal irreducible representations of $GSp(4,F_v)$ at non-archimedean local number fields $F_v$ the representation $\mathcal{F}_{\mathfrak{p}_v}(\pi,V_{\pi})$ has been determined explicitly \citep[Tab.\ 2.2]{my_thesis}. The dimension of the representation $\mathcal{F}_{\mathfrak{p}_v}(\pi_v)$ has also been determined whenever $\pi_v=\theta_{-}(\sigma_v)$ is the anisotropic theta-lift of an irreducible representation $\sigma_v$ of $GSO_{4,0}(F_v)$ \citep[Thm.\ 3.41]{my_thesis}.
For non-cuspidal representations at non-archimedean places $v$ of odd residue characteristic this functor has previously been studied by Breeding \citep{Breedings_thesis}.

In Section~\ref{sec:locallift} we recall the local endoscopic $L$-packets of admissible irreducible representations of $GSp(4,F_v)$. For representations $\pi_v$ in a local endoscopic L-packet at a non-archimedean place $v$ we determine the representation $\mathcal{F}_{\mathfrak{p}_v}(\pi_v)$ (Prop.~\ref{Prop:Local_KInvariants}). Under the condition that $\sigma_{1,v}$ and $\sigma_{2,v}$ admit non-zero invariants under $K_0^{(1)}(\mathfrak{p}_v)\subseteq GL(2,\mathfrak{o}_v)$ this implies that $\pi_v$ admits non-zero invariants under the modified principal congruence subgroup $K'(\mathfrak{p}_v)$ by Cor.~\ref{Cor:Local_K'Invariants}.
The global endoscopic $L$-packets attached to a cuspidal automorphic irreducible representation $\sigma$ of $M(\IA)=GSO_{2,2}(\IA)$ are introduced in Section~\ref{sec:glob_endo}. They consist of cuspidal automorphic irreducible representations $\pi$ of $GSp(4,\IA)$ that satisfy
\begin{align*}L(\pi_v,s)=L(\sigma_{1,v},s)\cdot L(\sigma_{2,v},s)
\end{align*} at almost every place $v$. The local factors $\pi_v$ are then contained in the local $L$-packets attached to $\sigma_v$ \citep[Thm.\,5.2]{Weissauer200903}.
We determine the action of $GSp(4,\hat{\mathfrak{o}})$ on the invariant space $\pi^{K(S)}$ explicitly. For the modified principal congruence subgroup we obtain the corresponding global version (Theorem~\ref{thm:main-result}).
This can be applied to construct Yoshida lifts like in \citep[Thm.\ 7.2]{Yoshida_Lift} with principal congruence subgroup level $N=\lcm(N_1,N_2)$ attached to cuspidal newforms $f_i$ with squarefree level $N_i$. Tehrani \citep{Tehrani} has obtained related results, he has proved the existence of non-zero invariant vectors in $\pi$ under paramodular subgroups.


In Section~\ref{sec:SK-lift} we make the analogous computations for Saito-Kurokawa lifts in the sense of Piatetski-Shapiro \citep{PS-SaitoK}. Let $\sigma$ be a cuspidal automorphic irreducible representation of $GL(2,\IA)$ and let $S$ be a finite set of places $v$ where $\sigma_v$ is in the discrete series such that $(-1)^{\#S}=\epsilon(\sigma,1/2)$. Let $\sigma_S\in\Irr(GL(2,\IA))$ be the irreducible constituent of $\Ione_{\IA^\times}\times\Ione_{\IA^\times}$ such that $\sigma_{S,v}$ is in the discrete series if and only if $v\in S$. Via local theta-lifts one constructs an irreducible automorphic representation $\pi$ of $GSp(4,\IA)$ with spinor $L$-function 
\begin{align*}
 L(\pi_v,s)=L(\sigma_v,s)\cdot L(\sigma_{S,v},s)
\end{align*}
at almost every place $v$. This $\pi$ is cuspidal if and only if $S$ is non-empty or $L(\sigma,1/2)=0$. 
Under the condition that $\sigma_v$ admits non-zero invariants under the Iwahori subgroup of $GL(2,F_v)$, this implies that $\pi_v$ admits non-zero invariants under the first principal congruence subgroup $K'(\mathfrak{p}_v)\subseteq GSp(4,\mathfrak{o}_v)$. In a more classical language this means we can can control the principal congruence subgroup level of Saito-Kurokawa Lifts from elliptic cuspforms of squarefree level to Siegel cuspforms of genus two (Cor.~\ref{Cor:SK_application}).
We can completely determine the action of $GSp(4,\mathfrak{o}_v)$ on the space of invariants in $\pi_v$ under the principal congruence subgroup $K(\mathfrak{p}_v)$ (Prop.~\ref{prop:SK_invars}).

In Section~\ref{sec:inncoh} we consider the Shimura variety \begin{align*}X_K=GSp(4,\IQ)\backslash GSp(4,\IA)/KK_\infty\end{align*} over $F=\IQ$, where $K_\infty$ is the stabilizer of $i\cdot I_2\in\IH_2$ in the Siegel upper halfspace $\IH_2$ of genus two and $K\subseteq GSp(4,\hat{\IZ})$ is the principal congruence subgroup of level $N=2$. For the local system $\mathcal{V}_\lambda$ over $GSp(4,\IQ)$ with weight $\lambda=(\lambda_1,\lambda_2)$ we can decompose the endoscopic and the Saito-Kurokawa part of the inner cohomology
$H_!^\bullet(X_K,\mathcal{V}_\lambda)$
in terms of $Sp(4,\IZ/2\IZ)$-representations. This proves the conjectures made by Bergstr\"{o}m, Faber and van der Geer \citep{BFG}. Indeed, the first equation of Corollary~\ref{Cor:BFG_endo_conjs} together with the natural isomorphism
\begin{align*}H_{!}^{(3,0)}(X_{K},\mathcal{V}_\lambda)\cong \mathcal{S}^{(2)}_{(k_1-k_2,k_2)}(\Gamma^{(2)}(2))
\end{align*} proves Conjecture 6.4 of loc.\ cit. and that implies Conjecture 6.1. Note that Conjecture 6.1 asks for the equality of $L$-factors at every non-archimedean place, this is implied by Prop.~\ref{prop:local_L-factors}. The second equation of Corollary~\ref{Cor:BFG_endo_conjs} proves Conjecture 7.1.
The analogous question for Saito-Kurokawa lifts is answered by Corollary~\ref{Cor:CAP_cohom_lev2}, proving Conjecture 6.6 and Conjecture 7.4 of loc.\ cit. For the case $\lambda_1>\lambda_2$ Conjecture 7.2 of loc.\ cit.\ is answered by Corollary~\ref{Cor:strict_endoscopy}. For $\lambda_1=\lambda_2$ we have Corollary~\ref{Cor:strict_endoscopy_SK} describing the sum of the dimensions of Saito-Kurokawa parts in the inner cohomology.

During the work on this article, the author was supported by a doctoral fellowship of the Research Training Group ``Symmetrie, Geometrie und Arithmetik'' at the University of Heidelberg. For valuable hints and discussions it is a pleasure to thank R.~Weissauer and U.~Weselmann.
\pagebreak
\section{Preliminaries}
For a totally real number field $F/\IQ$ with ring of integers $\mathfrak{o}$ let $\IA=\prod'_v F_v$ be its adele ring. At the non-archimedean places $v$ let $\mathfrak{o}_v$ and $\mathfrak{p}_v$ denote the ring of $F_v$-integers and its maximal ideal. The compact ring $\prod_{v<\infty}\mathfrak{o}_v\subseteq \IA$ will also be denoted $\hat{\mathfrak{o}}$. A non-archimedean place is sometimes called \emph{odd} or \emph{even} depending on its residue characteristic being odd or even, respectively. The valuation character of $F_v^\times$ is also denoted by $\nu_v(x)=|x|_v$. Arbitrary nontrivial quadratic characters of $F_v^\times$ are denoted by $\xi$. We also write $\xi_u$ for the unramified and $\xi_t$ for one of the two tamely ramified quadratic characters, such a $\xi_t$ only exists for odd residue characteristic. The group of symplectic similitudes $GSp(2n)$ over $\mathfrak{o}$ is 
\begin{align*}
GSp(2n)=\left\{g\in \Mat_{2n\times 2n} \,|\,gJg^t=\lambda J\text{ for } \lambda\in GL(1)\right\}\text{ with } J=\left(\begin{smallmatrix}&I_n\\-I_n&\end{smallmatrix}\right)
\end{align*} with similitude character $\simi:g\mapsto\lambda$.
For each non-archimedean place $v$ let 
\begin{align*}K(\mathfrak{p}_v):=\left\{g\in GSp(4,\mathfrak{o}_v)\,|\, g\equiv I_{4}\mod\mathfrak{p}_v\right\}
=\ker(GSp(4,\mathfrak{o}_v)\twoheadrightarrow GSp(4,\mathfrak{o}_v/\mathfrak{p}_v))
\end{align*}
 be the first \emph{principal congruence subgroup} of $GSp(4,\mathfrak{o})$
and let \begin{align*}
K'(\mathfrak{p}_v):=\left\{g\in GSp(4,\mathfrak{o}_v)\,|\,\exists \lambda\in \mathfrak{o}_v^\times: g\equiv \left(\begin{smallmatrix}I_2&\\&\lambda I_2\end{smallmatrix}\right)\mod\mathfrak{p}_v\right\}
\end{align*}
be the first \emph{modified principal congruence subgroup}.
The modified principal congruence subgroups satisfy the requirement $\mathrm{sim}(K'(\mathfrak{p}_v))=\mathfrak{o}_v^\times$ for strong approximation.
The analogous subgroup of $GL(2,\mathfrak{o}_v)$ will be denoted $K^{(1)}(\mathfrak{p}_v)=\ker(GL(2,\mathfrak{o}_v)\twoheadrightarrow GL(2,\mathfrak{o}_v/\mathfrak{p}_v))$. 
For the classical notion of level in the sense of \citep{Casselmann_res_atk_leh} depends on the congruence subgroup 
$K_0^{(1)}(\mathfrak{p}_v^n)=\{g\in GL(2,\mathfrak{o}_v)\,|\,g\equiv\left(\begin{smallmatrix}\ast&\ast\\&\ast\end{smallmatrix}\right)\mod\mathfrak{p}_v^n\}$.

By a representation $\pi$ of $GSp(2n,F_v)$ at a non-archimedean $v$ we will always mean an admissible complex linear representation. The twist of $\pi$ by a character $\chi:F_v^\times\to\IC^\times$ will be denoted by $\chi\cdot\pi:=(\chi\circ\simi)\otimes\pi$.
For non-cuspidal irreducible representations of $GSp(4,F_v)$ we use the notation of Roberts and Schmidt \citep{Roberts-Schmidt}. The Steinberg representation of $GSp(2n,F_v)$ will be denoted $\St_{GSp(2n,F_v)}$ and for the trivial representation we write $\Ione_{GSp(2n,F_v)}$.
Suppose $(\pi,V_{\pi})$ is a representation of $GSp(4,F_v)$. The restriction of $\pi$ to $GSp(4,\mathfrak{o}_v)$ stabilizes the space $V_{\pi}^{K(\mathfrak{p}_v)}$ of $K(\mathfrak{p}_v)$-invariant vectors. This defines the representation 
\begin{align}\label{eq:def_F1}
\mathcal{F}_{\mathfrak{p}_v}(\pi,V_{\pi}):=(\pi|_{GSp(4,\mathfrak{o}_v)},V_{\pi}^{K(\mathfrak{p}_v)})
\end{align}
of the group $GSp(4,\IF_q)\cong GSp(4,\mathfrak{o}_v/\mathfrak{p}_v)$. The functor of parahoric restriction 
\begin{align}\mathcal{F}_{\mathfrak{p}_v}:\Rep(GSp(4,F_v))\to \Rep(GSp(4,\IF_q))\text{}
\end{align} is additive, exact and it commutes with parabolic induction \citep[Cor.\ 2.19]{my_thesis}. It maps cuspidal irreducible representations to cuspidal irreducible ones or to zero \citep[Thm.\ 2.23]{my_thesis}. For the non-cuspidal irreducible representation $\pi$ of $GSp(4,F_v)$ the representation $\mathcal{F}_{\mathfrak{p}_v}(\pi)$ has been determined  in \citep[Table 2.2]{my_thesis}. A representation $\pi$ is called \emph{spherical}, if it admits non-zero invariants under $GSp(2n,\mathfrak{o}_v)$.

For a generic irreducible representation $\sigma_{v}$ of $GL(2,F_v)$ with unramified central character $\omega_{v}$ recall that its \emph{level} is $\mathfrak{p}_v^n$ for the smallest $n\in\IZ_{\geq0}$ such that $\sigma_{v}$ admits a non-zero subspace of $K_0^{(1)}(\mathfrak{p}_v^n)$-invariants.
If the level of $\sigma_v$ is $\mathfrak{p}_v$, then $\sigma_v$ is either $\St_{GL(2,F_v)}$ or $\xi_u\cdot\St_{GL(2,F_v)}$ \citep[Prop.\ 5.21]{Gelbart}.
For each real place $v$, the discrete series representation of $PGL(2,F_v)$ with lowest weight $k\in2\IZ_{>0}$ is denoted $\mathcal{D}(k-1)$. 
There are two notions of $L$- and $\epsilon$-factors for irreducible representations of $GSp(4,F_v)$, the standard factors of degree $5$ and the spinor factors of degree $4$; we will use the spinor factors constructed by Piatetski-Shapiro \citep{PS97}.
 We also consider the group
\begin{align*}M:=GSO_{2,2}=GL(2)\times GL(2)/\Delta GL(1)\end{align*} over $\mathfrak{o}$,
where $\Delta GL(1)$ is the image of the antidiagonal embedding $GL(1)\hookrightarrow GL(2)\times GL(2)$. An irreducible representation $\sigma$ of $M(F_v)$ corresponds to a pair of two irreducible representations $\sigma_1,\sigma_2$ of $GL(2,F_v)$ with equal central character $\omega_{\sigma_1}=\omega_{\sigma_2}$. Let $\sigma^\ast:=(\sigma_2,\sigma_1)$ be the representation obtained by switching the two factors.

The set of isomorphism classes of cuspidal automorphic irreducible representations of $GSp(2n,\IA)$ is denoted $\mathcal{A}_0(GSp(2n,\IA))$. Irreducible representations $\pi,\pi'$ of $GSp(2n,\IA)$ are \emph{weakly equivalent}, if their local factors $\pi_v,\pi'_v$ are isomorphic at almost all places.
A cuspidal automorphic irreducible representation $\pi$ of $GSp(2n,\IA)$ is \emph{CAP}, if it is weakly equivalent to a constituent of a globally parabolically induced reprentation. Among the automorphic cuspidal irreducible representation of $GSp(4,\IA)$ the CAP representations are distinguished by the fact that their spinor $L$-function has poles \citep[Thm.\ 2.2]{PS-SaitoK}. The factorization of an irreducible representation $\pi$ of $GSp(4,\IA)$ into its archimedean and non-archimedean part is denoted by $\pi=\pi_{\infty}\otimes\pi_{\mathrm{fin}}$.

The finite field of order $q$ is denoted $\IF_q$. Irreducible representations of $GSp(4,\IF_q)$ with odd $q$ have been classified by Shinoda \citep{Shinoda}. For even $q$ there is an isomorphism $GSp(4,\IF_q)\cong Sp(4,\IF_q)\times\IF_q^\times$, so every irreducible representation $\rho$ of $GSp(4,\IF_q)$ is a product of its central character $\omega_\rho$ and its restriction to $Sp(4,\IF_q)$. For representations of $Sp(4,\IF_q)$ with even $q$ we use the notation of Enomoto \citep{Enomoto1972}.
%
In \citep{BFG} a different notation for irreducible representations of $Sp(4,\IF_2)$ is used: Fix an isomorphism $Sp(4,\IF_2)\cong \Sigma_6$ to the symmetric group in six letters and label the irreducible representations by partitions of $6$, which describe the corresponding Young-Tableaux \citep[\S110]{VanDerWaerden1964}. The dictionary is given by Table~\ref{tab:dic}. The correspondence is only unique up to the outer automorphism of $\Sigma_6$.
\begin{table}
\centering
\begin{footnotesize}
\begin{tabular}{ccccccc
ccccc
}
\toprule
$Sp(4,\IF_2)$  &$\theta_0$   & $\theta_1$   &$\theta_2$   &$\theta_3$    &$\theta_4$    &$\theta_5$  
    &$\chi_{5}(1)$& $\chi_{8}(1)$&$\chi_{9}(1)$&$\chi_{12}(1)$& $\chi_{13}(1)$            \\
$\Sigma_6$     &$[6]$        & $[4,2]$      &$[2^3]$      &$[5,1]$       &$[3,2,1]$     &$[1^6]$     
    &$[2^2,1^2]$  & $[3^2]$      &$[2,1^4]$    &$[4,1^2]$     &$[3,1^3]$                  \\
$\dim$         &$1$          & $9$          &$5$          &$5$           &$16$          &$1$         
    &$9$          & $5$          &$5$          &$10$          &$10$                       \\
\bottomrule
\end{tabular}
\end{footnotesize}
\caption{Irreducible representations of $Sp(4,\IF_2)\cong\Sigma_6$.\label{tab:dic}}
\end{table}

The space of elliptic cuspforms with congruence group $\Gamma$ and weight $k$ is denoted $\mathcal{S}_k(\Gamma)$. For the subset of newforms, which are eigenforms of the Hecke algebra, we write $\mathcal{S}_k(\Gamma)^{\mathrm{new}}$. The space of genus two Siegel cuspforms with congruence group $\Gamma\subseteq Sp(4,\IZ)$ and type $\Sym^{r}(\mathrm{std})\otimes\det^{k}$ is $\mathcal{S}^{(2)}_{(r,k)}(\Gamma)$.


\section{The Local Endoscopic Lift}\label{sec:locallift}
Fix a place $v$ of the totally real number field $F$. 
The local endoscopic character lift $r:\Groth(M(F_v))\to \Groth(GSp(4,F_v))$ is defined on Grothendieck groups of virtual representations \citep[p.\,152]{Weissauer200903}. For every preunitary generic irreducible representation $\sigma_v$ of $M(F_v)$, its endoscopic lift $r(\sigma_v)$ is a linear combination of finitely many irreducible constituents. The set of these constituents is the \emph{local $L$-packet} attached to $\sigma_v$ \citep[Def.\ 4.5]{Weissauer200903}. The local $L$-packets of $\sigma_v=(\sigma_{1,v},\sigma_{2,v})$ and ${\sigma_v}^\ast=(\sigma_{2,v},\sigma_{1,v})$ coincide \citep[Lem.\ 4.5]{Weissauer200903}.

\begin{Theorem}[Weissauer/Shelstad]\label{thm:local_L-packets}
Let $\sigma_v=(\sigma_{1,v},\sigma_{2,v})$ be a generic preunitary irreducible representation of $M(F_v)$. Fix even integers $r_1>r_2\geq2$ and let $(k_1,k_2)=(\tfrac{1}{2}(r_1+r_2),\tfrac{1}{2}(r_1-r_2+4))$.
\begin{enumerate}

\item[i)]  If $\sigma_{1,v}$ and $\sigma_{2,v}$ are both in the discrete series,  the local $L$-packet attached to $\sigma_v$ consists of a generic representation $\Pi_+(\sigma_v)$ and a non-generic representation $\Pi_-(\sigma_v)$. 
\begin{enumerate}
\item For non-archimedean $v$ the representation $\Pi_+(\sigma_v)=\theta_+(\sigma_v)$ is the isotropic theta-lift of $\sigma_v$ and $\Pi_-(\sigma_v)=\theta_-(\sigma_v)$ is the anisotropic theta-lift of $\sigma_v$. 
\item For archimedean $v$ suppose $\sigma_{i,v}$ is in the discrete series with weight $r_i$. The representation $\Pi_-(\sigma_v)$ is holomorphic with lowest $K$-type of weight $(k_1,k_2)$ and $\Pi_+(\sigma_v)$ is non-holomorphic with lowest $K$-type of weight $(k_1,2-k_2)$.
\end{enumerate}
\item[ii)] If $\sigma_{2,v}=\mu_{1}\times\mu_{2}$ is Borel induced from a pair of smooth characters $\mu_{1},\mu_2$ of $F_v^\times$, then the local $L$-packet attached to $\sigma_{v}$ contains only the generic irreducible representation 
$\Pi_{+}(\sigma_v)=(\mu_1^{-1}\cdot\sigma_{1,v})\rtimes\mu_1\text{.}$
\end{enumerate}
\end{Theorem}
\begin{proof}
For archimedean places i) is implied by Cor.\ 4.2 and Remark 4.7 in \citep{Weissauer200903} and the proof is due to Shelstad \citep{Shelstad_L_indist}. For the $K$-types see \citep[p.\ 68]{Weissauer_asterisque}.
For non-archimedean $v$ i) is Thm.\ 4.5 in \citep{Weissauer200903}. ii) is implied by Lemma 4.27 in \citep{Weissauer200903}.
\end{proof}
The isotropic theta-lift $\theta_+(\sigma_v)$ is constructed in \citep[\S1(e)]{Soudry_Piatetski_Weil_lift}, it is denoted by $\Pi(\sigma)$ there.
For the anisotropic theta-lift $\theta_-(\sigma_v)$ see \citep[Def.\ 4.7]{Weissauer200903}.
The local $L$-packets at non-archimedean $v$ are given explictly by \citep[Thm.\ 4.5]{Weissauer200903}, compare Table~\ref{tab:local_lift}. The Table contains all possible cases, since $\Pi_\pm(\sigma_{1,v},\sigma_{2,v})=\Pi_\pm(\sigma_{2,v},\sigma_{1,v})$.
\begin{table}
\begin{center}
 \begin{tabular}{llll}
\toprule
$GL(2,F_v)$   & $GL(2,F_v)$                             & $GSp(4,F_v)$                                       & $GSp(4,F_v)$\\
$\sigma_{1,v}$& $\sigma_{2,v}$                          & $\Pi_+(\sigma_v)=\theta_+(\sigma_v)$               & $\Pi_-(\sigma_v)=\theta_-(\sigma_v)$\\
\midrule
$\mu_1\times\mu_2$ &$\mu_3\times\mu_4$                  & $\mu_1\mu_3^{-1}\times\mu_2\mu_3^{-1}\rtimes\mu_3$ & --- \\
$\mu_1\cdot\St_{GL(2,F_v)} $& $\mu_3\times\mu_4$        & $\mu_1\mu_3^{-1}\cdot\St_{GL(2,F_v)}\rtimes\mu_3$  & ---\\
$\rho_1$     & $\mu_3\times\mu_4$                       & $\mu_3^{-1}\rho_1\rtimes\mu_3$                     & ---\\
$\xi\mu\cdot\St_{GL(2,F_v)}$ &$\mu\cdot\St_{GL(2,F_v)}$ & $\delta(\St_{GL(2,F_v)}\nu_v^{1/2}\rtimes\mu\nu_v^{-1/2})$&  cuspidal\\
$\mu\cdot\St_{GL(2,F_v)}$ & $\mu\cdot\St_{GL(2,F_v)}$   & $\tau(S,\nu_v^{-1/2}\mu)$                          & $\tau(T,\nu_v^{-1/2}\mu)$\\
$\rho_1$     & $\mu\cdot\St_{GL(2,F_v)}$                & $\delta(\mu^{-1}\nu_v^{1/2}\rho_1\rtimes\mu\nu_v^{-1/2})$ &  cuspidal\\
$\rho_1$     & $\rho_2(\cong\rho_1)$                    & $\tau(S,\rho_1)$                                   & $\tau(T,\rho_1)$\\
$\rho_1$     & $\rho_2(\not\cong\rho_1)$                & cuspidal                                          & cuspidal\\
\bottomrule
 \end{tabular}
\caption{Local $L$-packets attached to generic irreducible representations $\sigma_v$ of $M(F_v)$ at non-archimedean $v$.\label{tab:local_lift}}
\end{center}
\end{table}


\subsection{Invariant Vectors for local $L$-packets}\label{sec:invar_vectors}
Fix a non-archimedean place $v$  
with residue field $\IF_q\cong\mathfrak{o}_v/\mathfrak{p}_v$ for the rest of this section. Recall that $K(\mathfrak{p}_v)\subseteq GSp(4,\mathfrak{o}_v)$ and $K^{(1)}(\mathfrak{p}_v)\subseteq GL(2,\mathfrak{o}_v)$ are the first principal congruence subgroups. Suppose $\pi_v$ is an irreducible representation of $GSp(4,F_v)$ in the local $L$-packet attached to a preunitary generic irreducible representation $(\sigma_v, V_{\sigma_v})$ of $M(F_v)$.
In this section, the space $\pi_v^{K(\mathfrak{p}_v)}$ of invariant vectors under the principal congruence subgroup $K(\mathfrak{p}_v)\subseteq GSp(4,\mathfrak{o})$ will be determined.

\begin{Prop}\label{Prop:Local_KInvariants} Let $\sigma_v$ be a generic irreducible representation of $M(F_v)$ and let $\pi_v$ be a constituent of the local $L$-packet of $GSp(4,F_v)$ attached to $\sigma_v$. Then
\begin{align*}
\sigma_{1,v}^{K^{(1)}(\mathfrak{p}_v)}\neq0\text{ and }\sigma_{2,v}^{K^{(1)}(\mathfrak{p}_v)}\neq0& & \Longrightarrow& & &\mathcal{F}_{\mathfrak{p}_v}(\pi_v)\text{ is given by Table~\ref{tab:invar_local_endo},}\\
\sigma_{1,v}^{K^{(1)}(\mathfrak{p}_v)}=0\text{ or } \sigma_{2,v}^{K^{(1)}(\mathfrak{p}_v)}=0&       & \Longrightarrow& & &\mathcal{F}_{\mathfrak{p}_v}(\pi_v)=0\text{.}
\end{align*}
 %
\end{Prop}
\begin{proof}
See \citep[Tables 4.1 and 4.2 and Thm.\,4.11]{my_thesis}. 
\end{proof}

\begin{Not}[Tables~\ref{tab:local_lift} and \ref{tab:invar_local_endo}] 
For $j=1,2,3,4$ let $\mu_j$ be smooth characters of $F_v^\times$ and $\rho_1$ and $\rho_2$ be cuspidal irreducible representations of $GL(2,F_v)$. The condition $\omega_{\sigma_1}=\omega_{\sigma_2}$ is implicitly imposed.

For Table \ref{tab:invar_local_endo} we additionally assume that $\sigma_{i,v}$ admits non-zero invariants under the first principal congruence subgroup $K^{(1)}(\mathfrak{p}_v)\subseteq GL(2,\mathfrak{o}_v)$.
That means the $\mu_j$ are at most tamely ramified and the restriction of $\mu_j$ to $\mathfrak{o}_v^\times$ defines a character $\tilde{\mu}_j$ of the residue field $\mathfrak{o}_v^\times/(1+\mathfrak{p}_v)\cong\IF_{q}^\times$.
To the cuspidal irreducible representations $\rho_i$ of $GL(2,F)$ with non-zero ${K^{(1)}(\mathfrak{p}_v)}$-invariants we attach a character $\Lambda_i$ of $\IF_{q^2}^\times$ as follows:
The action of $GL(2,\mathfrak{o}_v)$ on $\rho_{i}^{K^{(1)}(\mathfrak{p}_v)}$ defines a cuspidal irreducible representation $\tilde{\rho}_i=\mathcal{F}_{\mathfrak{p}_v}(\rho_i)$ of $GL(2,\IF_q)$ \citep[Thm.\ 2.23]{my_thesis}. This $\tilde{\rho}_i$ is parametrized by a character $\Lambda_i$ of $\IF_{q^2}^\times$, unique up to conjugation, such that  $\tr(\tilde{\rho}_i)\left(\begin{smallmatrix}0&-\alpha^{q+1}\\1&\alpha+\alpha^q\end{smallmatrix}\right)=-\Lambda_i(\alpha)-\Lambda_i^q(\alpha)$ for $\alpha\in\IF_{q^2}^\times-\IF_q^\times$.
Each character $\Lambda$ of $\IF_{q^2}^\times$ with $\Lambda^{q+1}=1$ factors over a character $\omega_{\Lambda}$ of the kernel of the norm $\nr_{\IF_{q^2}/\IF_q}:\IF_q^2\to\IF_q$ such that $\Lambda(x)=\omega_{\Lambda}(x^{q-1})$.
The quadratic character $\xi$ can be either $\xi_u$ or $\xi_t$.

For even $q$ fix an injective character $\hat{\theta}:\IF_{q^2}^\times\to \IC^\times$ and let $\hat{\gamma}$ and $\hat{\eta}$ be its restriction to $\IF_q^\times$ and $\ker\nr_{\IF_{q^2}^\times/\IF_q^\times}$, respectively. Let $k_j\in\IZ/(q-1)\IZ$ be such that $\hat{\gamma}^{k_j}=\tilde{\mu_j}$. We consider the natural embedding and projection 
\begin{align*}
\kappa:\IZ/(q+1)\IZ&\hookrightarrow \IZ/(q^2-1)\IZ\text{,} & \text{ and }& &  \kappa^{\ast}:\IZ/(q^2-1)\IZ&\twoheadrightarrow \IZ/(q+1)\IZ\text{,}\\
x&\mapsto(q-1)x & & & x&\mapsto x\text{.}  
\end{align*}
For the cuspidal representations let $l_i\in \IZ/(q^2-1)\IZ$ be such that $\hat{\theta}^{l_i}=\Lambda_i$.
If the central character of $\tilde{\rho}_i$ is trivial, $\kappa^{-1}(l_i)$ is well-defined and satisfies $\omega_{\Lambda_i}=\hat{\eta}^{\kappa^{-1}(l_i)}$.
Finally, for $\Lambda_1|_{\IF_q^\times}=\Lambda_2|_{\IF_q^\times}$ let $\tilde{k}_1=\tfrac{q+2}{2}(\kappa^{\ast}(l_1+l_2))$ and $\tilde{k}_2=\tfrac{q+2}{2}(\kappa^\ast(l_1-l_2))$.
\end{Not}

\begin{table}
\begin{footnotesize}
\centering
  \begin{tabular}{lllrrr}
\toprule
\multicolumn{2}{c}{$\sigma_v\in \Irr(M(F_v))$}&&\multicolumn{2}{c}{$\tilde{\pi}_v=\mathcal{F}_{\mathfrak{p}_v}\circ\Pi_{\epsilon_v}(\sigma_v)\in \Rep(GSp(4,\IF_q))$}\\
$\sigma_{1,v}$&$\sigma_{2,v}$&$\epsilon_v$&$\tilde{\pi}_v|_{Sp(4)}$ for even $q$ &$\tilde{\pi}_v$ for odd $q$ & $\dim\tilde{\pi}_v$\\
\midrule
$\mu_1\times\mu_2$      & $\mu_3\times\mu_4$        &$+$& $\chi_1(k_1-k_3,k_2-k_3)$           &$X_1(\tilde{\mu}_3/\tilde{\mu}_1,\tilde{\mu}_4/\tilde{\mu}_1,\tilde{\mu}_1)$& $(q+1)^2(q^2+1)$   \\
$\mu_1\times\mu_2$      & $\mu_3\cdot\St_{GL(2)}$   &$+$& $\chi_{10}(k_3-k_1)$                &$\chi_4(\tilde{\mu}_3/\tilde{\mu}_1,\tilde{\mu}_1)$& $(q^2+q)(q^2+1)$ \\               
$\mu_1\cdot\St_{GL(2)}$ & $\mu_1\cdot\St_{GL(2)}$   &$+$& $\theta_1+\theta_4$                 &$\theta_1(\tilde{\mu}_1)+\theta_5(\tilde{\mu}_1)$ & $q^4+q(q+1)^2/2$  \\
$\mu_1\cdot\St_{GL(2)}$ & $\mu_1\xi_u\cdot\St_{GL(2)}$&$+$& $\theta_3+\theta_4$                 &$\theta_4(\tilde{\mu}_1)+\theta_5(\tilde{\mu}_1)$ & $q^4+q(q^2+1)/2$  \\
$\mu_1\cdot\St_{GL(2)}$ &$\mu_1\xi_t\cdot\St_{GL(2)}$&$+$& ---                                 &$\tau_3(\tilde{\mu}_1)$                           & $q^4+q^2$         \\
$\mu_1\times\mu_2$      & $\mu_1\cdot\rho_2$        &$+$& $\chi_2(l_2)$                   &$X_2(\Lambda_2,\tilde{\mu}_1)$                    & $q^4-1$           \\
$\mu_1\cdot\St_{GL(2)}$ & $\mu_1\cdot\rho_2$        &$+$& $\chi_{12}(\kappa^{-1}(l_2))$   &$\chi_6(\omega_{\Lambda_2},\tilde{\mu}_1)$        & $(q^2+1)(q^2-q)$  \\
$\rho_1$       & $\rho_2$ ($\cong\rho_{1}$)         &$+$& $\chi_{13}(\kappa^{\ast}(l_1))$ &$\chi_8(\Lambda_1)$                               & $(q^2+1)(q^2-q)$  \\
$\rho_1$   & $\rho_2$ ($\not\cong\rho_{1}$)         &$+$& $\chi_{4}(\tilde{k}_1,\tilde{k}_2)$ &$X_5(\Lambda_1,\omega_{\Lambda_2/\Lambda_1})$     & $(q^2+1)(q-1)^2$  \\
\midrule
$\mu_1\cdot\St_{GL(2)}$ &$\mu_1\cdot\St_{GL(2)}$    &$-$& $\theta_2$                          &$\theta_3(\tilde{\mu})$                           & $q(q^2+1)/2$      \\
$\mu_1\cdot\St_{GL(2)}$ &$\mu_1\xi_u\cdot\St_{GL(2)}$ &$-$& $\theta_5$                          &$\theta_2(\tilde{\mu})$                           & $q(q-1)^2/2$      \\
$\mu_1\cdot\St_{GL(2)}$ &$\mu_1\xi_t\cdot\St_{GL(2)}$&$-$& ---                                 &$0$                                               & $0$               \\
$\mu_1\cdot\St_{GL(2)}$ &$\rho_2$                   &$-$& $0$                                 &$0$                                               & $0$               \\
$\rho_1$                &$\rho_2$ ($\cong\rho_{1}$) &$-$& $\chi_{9}(\kappa^{\ast}(l_1))$  &$\chi_7(\Lambda_1)$                               & $(q^2+1)(q-1)$    \\
$\rho_1$             &$\rho_2$ ($\not\cong\rho_{1}$)&$-$& $0$                                 &$0$                                               & $0$               \\
\bottomrule
\end{tabular}
\caption{Invariants of $\Pi_\pm(\sigma_v)$ under $K(\mathfrak{p}_v)$ for $\sigma_{i,v}^{K^{(1)}(\mathfrak{p}_v)}\neq0$.\label{tab:invar_local_endo}}
\end{footnotesize}
\end{table}

\begin{Corollary}\label{Cor:Local_K'Invariants}
For a preunitary generic irreducible representation $\sigma_v$ of $M(F_v)$ let $\pi_v=\Pi_\pm(\sigma_{1,v},\sigma_{2,v})$ be in the local $L$-packet attached to $\sigma_v$. Then we have:
\begin{align*}
&\sigma_{1,v}\text{ and }\sigma_{2,v}\text{ spherical} & \Longleftrightarrow& & &\pi_v\text{ spherical,}\\
&\sigma_{1,v}^{K_0^{(1)}(\mathfrak{p}_v)}\neq0\text{ and }\sigma_{2,v}^{K_0^{(1)}(\mathfrak{p}_v)}\neq0 & \Longrightarrow& & &\pi_v^{K(\mathfrak{p}_v)}\neq0\text{,}\\
&\pi_v^{K(\mathfrak{p}_v)}\neq0 & \Longleftrightarrow& & &\pi_v^{K'(\mathfrak{p}_v)}\neq0\text{.}
\end{align*}

\end{Corollary}
\begin{proof}
In the first line of Table \ref{tab:invar_local_endo}, $\pi_v$ is spherical if and only if the $\mu_j$ are all unramified. The other cases are not spherical, compare \citep[Tables 1.3 and 1.7]{my_thesis}.

$\sigma_{i,v}$ admits non-zero invariants under $K_0(\mathfrak{p}_v)$ if and only if $\sigma_{i,v}$ is either unramified or isomorphic to $\mu\cdot\St_{GL(2)}$ for an unramified character $\mu$ of $F^\times$. For the corresponding cases Table \ref{tab:invar_local_endo} implies $\pi_v^{K(\mathfrak{p}_v)}\neq0$.

The subspace of $K'(\mathfrak{p}_v)$-invariants in $\pi_v$ is the subspace of $\{\diag(1,1,\ast,\ast)\}$-invariants in $\tilde{\pi}_v$, this  is determined in \citep[Tables 1.4 and 1.8]{my_thesis}.
\end{proof}

\section{The Global Weak Endoscopic Lift}\label{sec:glob_endo}

In this section the weak endoscopic lift is recalled and its invariant spaces under principal congruence subgroups of squarefree level are described. We fix a generic cuspidal automorphic irreducible representation $\sigma=(\sigma_1,\sigma_2)\in\mathcal{A}_0(M(\IA))$ with $\sigma_1\not\cong\sigma_2$.
\begin{Def}\label{def:endolift}
A cuspidal automorphic irreducible representation $\pi\in\mathcal{A}_0(GSp(4,\IA))$ is called a \emph{weak endoscopic lift attached to $\sigma$} if for almost every place $v$ its local spinor $L$-factor is
\begin{align}\label{L_factor_equation}
 L(\pi_v,s)=L(\sigma_{1,v},s)L(\sigma_{2,v},s)\text{.}
\end{align}
The set of isomorphism classes of weak endoscopic lifts attached to $\sigma$ is the \emph{global $L$-packet attached to $\sigma$}.
\end{Def}
The representation $\pi_+(\sigma):=\bigotimes_v\Pi_+(\sigma_v)$ is a weak endoscopic lift attached to $\sigma$, so the $L$-packet is not empty \citep[Thm.\ 4.3, p.\ 194]{Weissauer200903}. 
A weak endoscopic lift attached to $\sigma$ is also attached to $\sigma^\ast=(\sigma_2,\sigma_1)$, but to no other cuspidal automorphic irreducible representation of $M(\IA)$ \citep[Prop.\ 5.2]{Weissauer200903}.
\begin{Theorem}[Weissauer]\label{thm:Weissauer_main_thm}
Suppose $\sigma=(\sigma_1,\sigma_2)\in\mathcal{A}_0(M(\IA))$ is an automorphic cuspidal irreducible representation of $M(\IA)$ with $\sigma_1\not\cong\sigma_2$. An irreducible representation $\pi=\bigotimes_v\pi_v$ of $GSp(4,\IA)$ belongs to the global $L$-packet attached to $\sigma$ if and only if
\begin{enumerate}
\item[(i)] $\pi_v\cong\Pi_-(\sigma_v)$ for an even number of places $v$ where $\sigma_v$ is in the discrete series,
\item[(ii)] $\pi_v\cong\Pi_+(\sigma_v)$ for every other place $v$.
\end{enumerate}
Then $\pi$ occurs with multiplicity $1$ in the discrete spectrum and is cuspidal, but not CAP.
\end{Theorem}
\begin{proof}
Any weak endoscopic lift $\pi$ attached to $\sigma$ is not CAP by \citep[Lem.\ 5.2]{Weissauer200903} and so we can apply \citep[Thm.\ 5.2]{Weissauer200903} to show the other properties. Conversely, an irreducible representation $\pi$ of $GSp(4,\IA)$
satisfying (i) and (ii) is weakly equivalent to $\pi_+(\sigma)$. Therefore $\pi$ must be cuspidal and not CAP, else $\pi_+(\sigma)$ would be CAP. Using Table~\ref{tab:local_lift} one can show that \eqref{L_factor_equation} holds at least for every non-archimedean place $v$ where $\sigma_v$ is spherical. Finally, by \citep[Thm.\ 5.2.4]{Weissauer200903} $\pi$ is automorphic with multiplicity one. Therefore $\pi$ is a weak endoscopic lift attached to $\sigma$.
\end{proof}
The global $L$-packet attached to $\sigma$ is equal to the set of equivalence classes of automorphic irreducible representations of $GSp(4,\IA)$ which are weakly equivalent to $\pi_+(\sigma)$, compare \citep[Def.\ 5.2]{Weissauer200903}. Let $d(\sigma)\in\IZ_{\geq0}$ be the number of places where $\sigma_v$ is in the discrete series. If $d(\sigma)\leq1$, then the global $L$-packet attached to $\sigma$ contains only $\pi_+$. For $d(\sigma)\geq2$ the global $L$-packet contains $2^{d(\sigma)-1}$ isomorphism classes of representations, compare \citep[Thm.\ 5.2.5]{Weissauer200903}. Indeed, by Thm.~\ref{thm:Weissauer_main_thm} any choice of an even number of places $v$ where $\sigma_v$ is in the discrete series gives rise to a weak endoscopic lift.

\subsection{Invariant Vectors for weak endoscopic Lifts}
For automorphic irreducible representations of $GL(2,\IA)$ and $GSp(4,\IA)$ we now look at invariant spaces under the groups
\begin{align*}
K_0^{(1)}(S)        &=\{g\in GL(2,\hat{\mathfrak{o}})\,|\, g_v\in K_0^{(1)}(\mathfrak{p}_v)\;\forall v\in S\}\text{,}\\
K^{(1)}(S)  &=\{g\in GL(2,\hat{\mathfrak{o}})\,|\, g_v\in K^{(1)}(\mathfrak{p}_v)\;\forall v\in S\} \text{,}\\
K(S)          &=\{g\in GSp(4,\hat{\mathfrak{o}})\,|\, g_v\in K(\mathfrak{p}_v)\;\forall v\in S\}        \text{,}\\
K'(S)         &=\{g\in GSp(4,\hat{\mathfrak{o}})\,|\, g_v\in K'(\mathfrak{p}_v)\;\forall v\in S\}       \text{}
\end{align*}
for finite sets $S$ of non-archimedean places of $F$.
\begin{Theorem}\label{thm:main-result}
Suppose $\sigma=(\sigma_1,\sigma_2)\in\mathcal{A}_0(M(\IA))$ is a cuspidal automorphic irreducible representation of $M(\IA)$ with $\sigma_1\not\cong\sigma_2$.
Let $\pi=\pi_\infty\otimes\pi_{\mathrm{fin}}\in\mathcal{A}_0(GSp(4,\IA))$ be a cuspidal automorphic irreducible representation in the global $L$-packet attached to $\sigma$. Let $S=S_1\cup S_2$ be finite sets of non-archimedean places.
\begin{enumerate} 
\item[(i)] If $\sigma_{1,\mathrm{fin}}^{K^{(1)}(S_1)}\neq0$ and $\sigma_{2,\mathrm{fin}}^{K^{(1)}(S_2)}\neq0$, then the action of $GSp(4,\hat{\mathfrak{o}})$ on $\pi_{\mathrm{fin}}^{K(S)}$ is isomorphic to the representation $\bigotimes_{v\in S}\mathcal{F}_{\mathfrak{p}_v}(\pi_v)$ as given by Table~\ref{tab:invar_local_endo}. Especially, for $\sigma_{1,\mathrm{fin}}^{K_0^{(1)}(S_1)}\neq0$ and $\sigma_{2,\mathrm{fin}}^{K_0^{(1)}(S_2)}\neq0$ we have $\pi_{\mathrm{fin}}^{K'(S)}\neq0$.
\item[(ii)] If $\sigma_{1,\mathrm{fin}}^{K^{(1)}(S_1)}=0$ or $\sigma_{2,\mathrm{fin}}^{K^{(1)}(S_2)}=0$, then $\pi_{\mathrm{fin}}^{K(S)}=0$.
\end{enumerate}
\end{Theorem}
\begin{proof}
For each non-archimedean place $v$ the local factors $\pi_v$ are given by Thm.~\ref{thm:Weissauer_main_thm}. Recall that $\dim\pi_v^{GSp(4,\mathfrak{o}_v)}=1$ for spherical $\pi_v$ and so these factors can be dropped from the tensor product $\pi_{\mathrm{fin}}^{K(S)}$. Now Cor.~\ref{Cor:Local_K'Invariants} and Prop.~\ref{Prop:Local_KInvariants} imply (i) and (ii).
\end{proof}

\begin{Prop}\label{prop:local_L-factors}
Let $F=\IQ$. Suppose $\sigma=(\sigma_1,\sigma_2)$ is a cuspidal automorphic irreducible representation of $M(\IA)$ with $\sigma_1\not\cong\sigma_2$ and trivial central character $\omega_{\sigma}$. Fix a cuspidal automorphic irreducible representation $\pi$ of $GSp(4,\IA)$ in the global $L$-packet attached to $\sigma$. Then for every nonarchimedean place $v$ we have:
\begin{align}\label{eq:L-eps-equat}
 L(\sigma_{1,v},s)L(\sigma_{2,v},s)=L(\pi_v,s)\quad\text{ and }\quad \epsilon(\sigma_{1,v},s)\epsilon(\sigma_{2,v},s)=\epsilon(\pi_v,s)\text{.}
\end{align}
\end{Prop}

\begin{proof} Each non-archimedean factor $\pi_v$  is given by Table~\ref{tab:local_lift}.
For non-cuspidal $\pi_v$ the local $L$- and $\epsilon$-factors are given by Roberts and Schmidt \citep[Tables A.8 and A.9]{Roberts-Schmidt}. 
Now assume $\pi_v\cong\theta_+(\sigma_v)$ is generic and cuspidal, then $\sigma_{1,v}$ and $\sigma_{2,v}$ are both cuspidal and so $L(\pi_v,s)=1=L(\sigma_1,s)L(\sigma_2,s)$ by \citep[Prop.\,3.9]{Takloo_L}. The corresponding equation of $\gamma$-factors has been shown by Soudry and Piatetski-Shapiro \citep[Thm.\ 3.1]{Soudry_Piatetski_Weil_lift} and this implies \eqref{eq:L-eps-equat} by the local functional equation.
Next, for non-generic and cuspidal $\pi_v\cong\theta_-(\sigma_v)$ at a place $v$ with odd residue characteristic, \eqref{eq:L-eps-equat} is implied by \citep[Thm.\ 4.4 and Cor.\ 4.5]{Danisman_thesis} and by the fact the local Jacquet-Langlands-correspondence preserves $L$- and $\epsilon$-factors.

By Riemann-Roch there are cuspidal automorphic irreducible representations $\sigma'_1\not\cong\sigma'_2$ of $GL(2,\IA)$, unramified at $v=2$ but both ramified at some other non-archimedean place, whose archimedean factors are $\sigma'_{i,\infty}\cong\sigma_{i,\infty}$. Now the previous arguments apply to $\sigma'_i$ at every non-archimedean place. Hence Eq.\,\eqref{eq:L-eps-equat} holds for any lift $\pi'$ of $\sigma$ with $\pi_\infty\cong\pi'_\infty$.
By the functional equation for $\sigma'_i$ and $\pi'$, the local $\gamma$-factors must satisfy $\gamma(\sigma_{1,\infty},s)\gamma(\sigma_{2,\infty},s)=\gamma(\pi_\infty,s)$, and by the one for $\sigma$ and $\pi$ we have then $\gamma(\sigma_{1,2},s)\gamma(\sigma_{2,2},s)=\gamma(\pi_2,s)$.
The implication (\citep[Thm.\ 4.4]{Danisman_thesis} $\Rightarrow$ \citep[Cor.\ 4.5]{Danisman_thesis}) holds for $v=2$, so the $\gamma$-factors uniquely determine the $L$- and $\epsilon$-factors of $\pi_2$.
\end{proof}

The weak endoscopic lift defines a lifting from pairs of elliptic cuspidal eigenforms to Siegel cuspforms of genus two. There is already a similar result by R.~Schmidt and A.~Saha \citep[Prop.\ 3.1]{Schmidt-Saha} under the restriction that the Atkin-Lehner-eigenvalues of $f_1$ and $f_2$ coincide at every place dividing $\gcd(N_1,N_2)$.
\begin{Corollary}[Yoshida-Lift]\label{Cor:Yoshida}
For $i=1,2$ suppose $f_i\in \mathcal{S}_{r_i}(\Gamma_0(N_i))$ are elliptic newforms with squarefree $N_i$ and weight $r_1>r_2\geq2$, eigenforms under the Hecke-algebra, such that $\gcd(N_1,N_2)>1$. Then for $N=\lcm(N_1,N_2)$ there is a genus two Siegel cuspidal eigenform $f$ with type $\Sym^{r_2-2}(\mathrm{std})\otimes\det^{(r_1-r_2)/2+2}$, congruence group $\Gamma^{(2)}(N)$ and with spinor L-factor 
\begin{align}\label{eq:sqfree-mod-lift}
L_p(f,s)=L_p(f_1,s)L_p(f_2,s+\tfrac{1}{2}(r_2-r_1))\;\;\forall p<\infty\text{.}
\end{align}
\end{Corollary}
\begin{proof}
For $F=\IQ$ let $\sigma_i$ be the cuspidal automorphic irreducible representation of $GL(2,\IA)$ generated by $f_i$ and let $S_i=\{v<\infty\,|\, \mathfrak{p}_v \ni N_i\}$ be the set of non-archimedean places $v$ where $\sigma_{i,v}$ is in the discrete series. Fix a prime $p_0$ dividing $N_1$ and $N_2$, then by Thm.~\ref{thm:main-result} there is a weak endoscopic lift $\pi$, attached to $\sigma$, that is locally generic at every place except $p_0$ and $\infty$.
By Thm.~\ref{thm:local_L-packets} the archimedean factor $\pi_\infty=\Pi_-(\sigma_\infty)$ is holomorphic. We have $\sigma_{i,\mathrm{fin}}^{K_0^{(1)}(S_i)}\neq0$ by strong approximation, so Thm.~\ref{thm:main-result} implies $\pi_{\mathrm{fin}}^{K'(S)}\neq0$ for $S=S_1\cup S_2$.
Hence there is a non-zero adelic automorphic form $\phi$ in $\pi$ invariant under $K'(S)$. By strong approximation $\phi$ corresponds to a Siegel modular form $f$ of weight $\Sym^{r_2-2}(\mathrm{std})\otimes\det^{(r_1-r_2)/2+2}$ invariant under the principal congruence subgroup $K'(S)\cap Sp(4,\IQ)=\Gamma^{(2)}(N)$. Finally, Prop.~\ref{prop:local_L-factors} implies \eqref{eq:sqfree-mod-lift}.
\end{proof}

\section{The Saito-Kurokawa Lift}\label{sec:SK-lift}
The classical Saito-Kurokawa Lift has been constructed by Maass, Andrianov and Zagier, compare \citep{Zagier_Sur_SK}. To an elliptic cuspform $f\in \mathcal{S}_{2k-2}(SL(2,\IZ))$  with even $k\geq 10$ it attaches a scalar-valued Siegel cuspform $\tilde{f}\in \mathcal{S}^{(2)}_{(0,k)}(Sp(4,\IZ))$ with spinor $L$-function
\begin{align*}
 L(\tilde{f},s)=\zeta(s-k+1)\zeta(s-k+2)L(f,s)\text{,}
\end{align*} where $\zeta$ denotes the Riemann Zeta-Function.
A representation theoretic description of this lift in terms of the corresponding automorphic representations has been given by Piatetski-Shapiro \citep{PS-SaitoK} using results of Waldspurger \citep{Waldspurger-Shimura_quats}. In order to define the local theta-lifts, one has to fix a non-trivial additive character $\psi_v$ of $F_v$. The lift is indedependent of $\psi_v$ and so $\psi_v$ is suppressed in the notation.
For a given generic irreducible representation $\sigma_v$ of $PGL(2,F_v)$ we consider two possible cuspidal irreducible representations $\tilde{\pi}_v$ of the double cover $\widetilde{SL}(2,F_v)$ of $SL(2,F_v)$.
One possibility is $\tilde{\pi}_v=\theta_1'(\sigma_v)$ via the $\theta'_1$-lift between $PGL(2,F_v)$ and $\widetilde{SL}(2,F_v)$. For $\sigma_v$ in the discrete series there is another possible lift, $\tilde{\pi}_v=\theta'_2(\sigma_v^{JL})$. The Jacquet-Langlands-correspondence $\sigma_v\mapsto\sigma_v^{JL}$ gives rise to an irreducible representation of $PD^\times(F_v)$, the group of units in the quaternion division algebra over $F_v$ modulo center. Then $\sigma_v^{JL}$ corresponds to an irreducible representation $\tilde{\pi}_v$ via the $\theta'_2$-lift from $PD^\times(F_v)$ to $\widetilde{SL}(2,F_v)$. For both cases attached to $\tilde{\pi}_v$ there is a unique irreducible representation of $PGSp(4,F_v)$ via the $\theta$-correspondence described by Piatetski-Shapiro \citep{PS-SaitoK}.
These liftings can be visualized by the following non-commutative diagram of partial one-to-one functions:
\begin{align*}
\xymatrix{
 \Irr(PGL(2,F_v))\ar@{<->}[r]^-{\theta'_1}\ar@{<->}[d]_-{JL}    &  \Irr(\widetilde{SL}(2,F_v))\ar@{<->}[r]^-{\theta}   &    \Irr(PGSp(4,F_v))\text{.}\\
 \Irr(PD^\times(F_v))\ar@{<->}[ur]_-{\theta'_2}                        &                                                &
}
\end{align*}

An explicit construction of this lifting can be found in \citep{PS-SaitoK}, compare \citep{Schmidt_local_SK,Schmidt_SK_classical}. Fix a finite set $S$ of $F$-places where $\sigma_v$ is in the discrete series, and let $\sigma_{S,v}$ be \begin{align*}\sigma_{S,v}:=\begin{cases}\Ione_{GL(2,F_v)}&v\notin S\text{,}\\\St_{GL(2,F_v)}& v\in S\text{.}\end{cases}\end{align*}
Then $\sigma_S=\bigotimes_v\sigma_{S,v}$ is a constituent of the parabolically induced representation $\Ione_{\IA^\times}\times\Ione_{\IA^\times}$ of $GL(2,\IA)$. Since $\sigma_S$ has trivial central character, it defines a non-cuspidal automorphic irreducible representation of $PGL(2,\IA)$.
\begin{Def}
 For a generic cuspidal automorphic irreducible representation $\sigma$ of $PGL(2,\IA)$ and for $S$ as above, the \emph{local Saito-Kurokawa Lift at $v$} is the irreducible representation of $PGSp(4,F_v)$ given by
\begin{align*}
\Pi(\sigma_v,\sigma_{S,v}):=
\begin{cases}
\theta\circ\theta'_1(\sigma_v)& v\notin S\text{,}\\\theta\circ\theta'_2(\sigma_v^{JL})& v\in S\text{.}
\end{cases}
\end{align*} The \emph{global Saito-Kurokawa Lift} is $\Pi(\sigma,\sigma_S)=\bigotimes_v\Pi(\sigma_v,\sigma_{S,v})$.
\end{Def}
The local factors $\pi_v=\Pi(\sigma_v,\sigma_{S,v})$ are given in \citep[p.~24]{Schmidt_local_SK} and also in Table~\ref{tab_SK}. The lift $\theta\circ\theta'_2(\sigma_v^{JL})$ coincides with the anisotropic lift $\theta_-(\sigma_v,\St)$ \citep[Prop.~5.8)]{Schmidt_local_SK}. Each factor $\Pi(\sigma_v,\sigma_{S,v})$ is non-generic.

\begin{Theorem}[Piatetski-Shapiro/Schmidt/Waldspurger]\label{thm:SK}
 Suppose $\sigma$ is an irreducible cuspidal automorphic representation of $PGL(2,\IA)$ and $S$ is a finite set of places where $\sigma_v$ is in the discrete series. If 
\begin{align}\label{eq:SK_cond}(-1)^{\#S}=\epsilon(\sigma,1/2)\text{,}\end{align}
then $\pi=\Pi(\sigma,\sigma_{S})$ is an irreducible automorphic representation of $PGSp(4,\IA)$ with spinor factors 
\begin{align}\label{eq:L_eq_SK}L(\pi_v,s)=L(\sigma_v,s)L(\sigma_{S,v},s) \quad\text{ and }\quad \epsilon(\pi_v,s)=\epsilon(\sigma_v,s)\epsilon(\sigma_{S,v},s)\end{align}
 at almost every $v$.
The representation $\pi$ is cuspidal if and only if $L(\sigma,1/2)=0$ or $S\neq\emptyset$.
\end{Theorem}
\begin{proof} The existence and automorphy of $\pi$ is Thm.~3.1 from \citep{Schmidt_local_SK}. The equation of $L$-functions is given by Remark 3.2 in \citep{Schmidt_local_SK} depending on the Local Langlands Correspondence \citep{Gan_Takeda_GSp4}. Cuspidality of $\pi$ is implied by Thm.~3.1 of Schmidt \citep{Schmidt_local_SK}, compare Thm.~2.6 of Piateski-Shapiro \citep{PS-SaitoK}.
\end{proof}

For non-archimedean places the local Saito-Kurokawa lift is given by the third column of Table \ref{tab_SK}. At the archimedean places for $k\geq3$ we have $\Pi(\mathcal{D}(2k-3),\St)=\theta_-(\nu^{1/2}\mathcal{D}(2k-3),\nu^{-1/2})$, a holomorphic discrete series representation with $K$-type of weight $(k,k)$, which appears in the cohomology with Hodge types $(3,0)$ and $(0,3)$. The non-tempered Langlands quotient $\Pi(\mathcal{D}(2k-3),\Ione_{PGL(2,F_v)})=L(\nu^{1/2}\mathcal{D}(2k-3),\nu^{-1/2})$ has with minimal $K$-type of weight $(k-1,1-k)$ and contributes to cohomology with Hodge-type $(1,1)$ and $(2,2)$ \citep[7.3.ii)]{Blasius_Rogawski_zeta}. This Langlands quotient has been denoted $\pi_\lambda^{2,\pm}$ in \citep{Taylor_Siegel_threefolds}.
For additional information on the Saito-Kurokawa-lift, see \citep[Sect.\ 4]{Schmidt_local_SK}.

Every global Saito-Kurokawa lift $\pi$ is weakly equivalent to a globally Siegel induced representation, this is directly implied by Table~\ref{tab_SK}. For $F=\IQ$ equation \eqref{eq:L_eq_SK} holds at every non-archimedean place $v$, the proof is analogous to Prop.~\ref{prop:local_L-factors}.
\begin{table}\centering
\begin{footnotesize}
 \begin{tabular}{lllrrr}
\toprule
$\sigma_v$      &$\sigma_{S,v}$&   $\pi_v=\Pi(\sigma_v,\sigma_{S,v})$  &\multicolumn{3}{c}{$\tilde{\pi}_v=\mathcal{F}_{\mathfrak{p}_v}(\pi_v)\in\Rep(PGSp(4,\IF_q))$}   \\ 
                   &           &                        & $(\tilde{\pi}_v)|_{Sp(4)}$ for even $q$ & $\tilde{\pi}_v$ for odd $q$ & $\dim\tilde{\pi}_v$ \\\midrule
$\mu\times\mu^{-1}$& $\Ione$   & $(\mu\cdot\Ione)\rtimes\mu^{-1}$& $\chi_6(k)$&$\chi_3(\tilde{\mu},\tilde{\mu}^{-1})$ &$(q+1)(q^2+1)$\\
$\St$              & $\Ione$   & $L(\nu^{1/2}\cdot\St,\nu^{-1/2})$       & $\theta_3$&$\theta_4(1)$&$q(q^2+1)/2$\\
$\St$              & $\St$     & $\tau(T,\nu^{-1/2})$                    & $\theta_2$&$\theta_3(1)$&$q(q^2+1)/2$\\
$\xi_u \cdot \St$    & $\Ione$   & $L(\nu^{1/2}\xi_u\cdot\St,\nu^{-1/2})$    & $\theta_1$&$\theta_1(1)$&$q(q+1)^2/2$\\
$\xi_u \cdot \St$    & $\St$     & $\theta_-(\xi_u \St,\St)$ cusp.           & $\theta_5$&$\theta_2(1)$&$q(q-1)^2/2$\\
$\xi_t \cdot \St$   & $\Ione$   & $L(\nu^{1/2}\xi_t\cdot\St,\nu^{-1/2})$   & $-$       &$\tau_2(1)  $&$q(q^2+1)$  \\
$\xi_t \cdot \St$   & $\St$     & $\theta_-(\xi_t\St,\St)$ cusp.           & $-$       &$0$          &$0$         \\
$\rho$             & $\Ione$   & $L(\nu^{1/2}\rho,\nu^{-1/2})$ & $\chi_8(\kappa^{-1}(l))$&$\chi_5(\omega_{\Lambda},1)$&$(q-1)(q^2+1)$\\
$\rho$             & $\St$     & $\theta_-(\rho,\St)$ cusp.              & $0$       &$0$          &$0$         \\
\bottomrule
 \end{tabular}
\caption{The local Saito-Kurokawa Lift at the non-archimedean places. For the three right columns assume that $\sigma_v^{K^{(1)}(\mathfrak{p}_v)}\neq0$.\label{tab_SK}}
\end{footnotesize}
\end{table}
\pagebreak
\subsection{Invariant Vectors for Saito-Kurokawa Lifts}
For the local Saito-Kurokawa lift we also have a statement about invariant vectors: 
\begin{Prop}\label{prop:SK_invars}
For a generic irreducible representation $\sigma_v$ of $PGL(2,F_v)$ at a non-archimedean place $v$ let 
$\pi_v$
 be a local Saito-Kurokawa lift of $\sigma$. 
Then:
\begin{align*}
&\sigma_v^{K^{(1)}(\mathfrak{p}_v)}\neq0                & & \Longrightarrow     & &\pi_v^{K(\mathfrak{p}_v)}\text{ is given by Table~\ref{tab_SK},}\\
&\sigma_v^{K^{(1)}(\mathfrak{p}_v)}=0 & &\Longrightarrow& & \pi_v^{K(\mathfrak{p}_v)}=0\text{,}\\
&\sigma_v \text{ spherical}                             & & \Longleftrightarrow & &\pi_v\text{ spherical,}\\
&\sigma_v^{K_0^{(1)}(\mathfrak{p}_v)}\neq0              & & \Longrightarrow     & &\pi_v^{K(\mathfrak{p}_v)}\neq0\text{,}\\
&\pi_v^{K(\mathfrak{p}_v)}\neq0 & &\Longleftrightarrow  & & \pi_v^{K'(\mathfrak{p}_v)}\neq0\text{.}
\end{align*}
\end{Prop}

\begin{proof}
For non-cuspidal $\pi_v$ the representation $\tilde{\pi}_v=\pi_v^{K(\mathfrak{p}_v)}$ is determined in \cite[Tab. 2.2]{my_thesis}. The cuspidal $\pi_v$ are given by the anisotropic theta-lift $\theta_-(\sigma_v,\St)$ and $\tilde{\pi}_v$ has been determined in \cite[Table 4.1]{my_thesis}.
Only the first line of Table \ref{tab_SK} for unramified $\mu$ contains spherical representations. 
Non-zero ${K_0^{(1)}(\mathfrak{p}_v)}$-invariants are only possible for $\sigma_v\cong\St$, for $\sigma_v\cong\xi_u\cdot\St$ and for spherical $\sigma_v$, so the Table implies $\pi^{K(\mathfrak{p}_v)}\neq0$.
For the dimension of $\diag(1,1,\ast,\ast)$-invariants in $\tilde{\pi}_v$, see \cite[Tables 1.4 and 1.8]{my_thesis}.
\end{proof}
The reader can easily construct the global statement analogous to Theorem \ref{thm:main-result}.
\begin{Not}[Table \ref{tab_SK}] The notation is analogous to Table~\ref{tab:invar_local_endo}. Representations of $PGSp(4)$ are representation of $GSp(4)$ with trivial central character. For an at most tamely ramified character $\mu$ of $F_v^\times$ we write $\tilde{\mu}$ for its restriction to $\mathfrak{o}_v^\times/(1+\mathfrak{p}_v)$. Each cuspidal irreducible representation $\rho$ of $GL(2,F_v)$ with $\rho^{K^{(1)}(\mathfrak{p}_v)}\neq0$ defines a cuspidal irreducible representation $\tilde{\rho}$ of $GL(2,\IF_q)$, which in turn defines a character $\Lambda$ of $\IF_{q^2}^\times$. The central character of $\rho$ is unramified, so $\Lambda(x)=\omega_\Lambda(x^{q-1})$ factors over a character $\omega_\Lambda$ of $\ker\nr_{\IF_{q^2}^\times/\IF_q^\times}$. By $\xi_u$ we denote the non-trivial unramified quadratic character of $F_v^\times$ and $\xi_t$ is one of the tamely ramified quadratic characters.
For even $q$ let $k\in \IZ/(q+1)\IZ$ be such that $\tilde{\mu}=\hat{\gamma}^k$ as an $\IF_q^\times$-character and let $l\in\IZ/(q^2-1)\IZ$ be such that $\Lambda=\hat{\theta}^{l}$, then $\omega_\Lambda=\hat{\eta}^{\kappa^{-1}(l)}$.
\end{Not}

Proposition \ref{prop:SK_invars} can be applied to determine the principal congruence subgroup level of Saito-Kurokawa lifts in the sense of Piatetski-Shapiro. The following corollary has already been shown by R.~Schmidt \citep[Thm.\ 5.2.ii)]{Schmidt_SK_classical} for even $k$ and $M=\prod\limits_{\epsilon_p=-1}p$.
\begin{Corollary}[Classical Saito-Kurokawa Lift]\label{Cor:SK_application}
 Suppose $f\in \mathcal{S}_{2k-2}(\Gamma_0(N))$ is an elliptic cuspidal newform of squarefree level $N$ and weight $2k-2\geq4$ with Atkin-Lehner-eigenvalues $\epsilon_p$ at $p\mid N$, an eigenform of the Hecke-algebra. For any divisor $M$ of $N$ with 
\begin{align}\label{eq:class_SK_cond_cusp}
 (-1)^{\#\{p\mid M\}}=(-1)^{k}\prod_{p\mid N} \epsilon_p\text{,}
\end{align}
there is a scalar-valued Siegel cuspform $\tilde{f}\in \mathcal{S}^{(2)}_{(0,k)}(\Gamma^{(2)}(N))$, an eigenform for the Hecke-algebra, with spinor $L$-function
\begin{align*}
 L(\tilde{f},s)=\zeta(s-k+1)\zeta(s-k+2)L(f,s)\prod_{p|M}\frac{(1-p^{-s+k-1})(1-p^{-s+k-2})}{(1+\epsilon_p p^{-s+k-2})} \text{,}
\end{align*} where $\zeta$ denotes the Riemann zeta function.
 \end{Corollary}
\begin{proof}
For $F=\IQ$ set $S:=\{p\mid M\}\cup\{\infty\}$ and let $\sigma$ be the cuspidal automorphic irreducible representation of $GL(2,\IA)$ generated by $f$. The lift $\pi=\Pi(\sigma,\sigma_S)$ is automorphic because of Eq.\ \eqref{eq:class_SK_cond_cusp} and $\epsilon(\sigma,1/2)=\epsilon(f,k-1)=(-1)^{k-1}\prod_{p\mid N}\epsilon_p$.
It is cuspidal because $S$ is not empty. Prop.~\ref{prop:SK_invars} implies $\pi_v^{K'((p))}\neq0$ for 
 $K'((p))=K'(\mathfrak{p}_v)$.
By strong approximation
 each $K'(\{p\mid N\})$-invariant Hecke-eigenvector in $\pi$ gives rise to a Siegel cuspidal eigenform $\tilde{f}$, invariant under $K'_N\cap Sp(4,\IZ)=\Gamma^{(2)}(N)$. Since $F=\IQ$, \eqref{eq:L_eq_SK} holds at every non-archimedean place by an argument similar to Prop.~\ref{prop:local_L-factors}.
\end{proof}

\section{The Inner Cohomology}\label{sec:inncoh}
From now on we only consider $F=\IQ$ and fix a neat compact open subgroup $K\subseteq GSp(4,\hat{\IZ})=\prod_{p<\infty}GSp(4,\IZ_p)$. Every principal congruence subgroup of principal congruence subgroup level $N\geq 3$ is neat. Consider the Shimura variety 
\begin{align*}X_K=GSp(4,\IQ)\backslash GSp(4,\IA)/K K_\infty\text{,}
\end{align*} where $K_\infty$ is the stabilizer of $iI_2\in \IH_2$ in the Siegel upper half plane.
Fix a pair of integers $\lambda=(\lambda_1,\lambda_2)$ with $\lambda_1\geq\lambda_2\geq0$. Let $V_\lambda$ be the finite dimensional representation of $GSp(4,\IQ)$ with central character $t\mapsto t^{\lambda_1+\lambda_2}$ whose restriction to $Sp(4,\IQ)$ has weight $(\lambda_1,\lambda_2)$. We consider the local coefficient system 
\begin{align*}
 \mathcal{V}_\lambda: GSp(4,\IQ)\backslash(GSp(4,\IA)/KK_\infty\times V_\lambda)\longrightarrow X_K\text{.} 
\end{align*}
For this local system let $H^\bullet(X_K,\mathcal{V}_\lambda)$ be the \'{e}tale cohomology and $H_c^\bullet(X_K,\mathcal{V}_\lambda)$ be the \'{e}tale cohomology with compact support. The inner cohomology is
\begin{align*}
 H_!^\bullet(X_K,\mathcal{V}_\lambda)=\mathrm{Image}(H_c^\bullet(X_K,\mathcal{V}_\lambda)\to H^\bullet(X_K,\mathcal{V}_\lambda))\text{,}
\end{align*} it is denoted $H_P^\bullet$ in \citep{Taylor_Siegel_threefolds}.
Attached to the projective limit $X:=\varprojlim\limits_K X_K$ over neat compact open subgroups there is the inner cohomology
\begin{align*}
 H_!^\bullet(X,\mathcal{V}_\lambda)=\varinjlim_K H_!^\bullet(X_K,\mathcal{V}_\lambda)\text{.}
\end{align*}
For neat $K$ we get back the inner cohomology of $X_K$ by taking $K$-invariants
\begin{align*} H_!^\bullet(X,\mathcal{V}_\lambda)^K=H_!^\bullet(X_K,\mathcal{V}_\lambda)\text{.}\end{align*} By this equation we can define $H_!^\bullet(X_K,\mathcal{V}_\lambda)$ for every compact open subgroup $K$ of $GSp(4,\hat{\IZ})$.
The cohomology in degree $3$ admits the following decomposition \citep{Taylor_Siegel_threefolds}: 
\begin{align}\label{eq:inner_cohom_XH}
 H^{3,0}_!(X,\mathcal{V}_{\lambda})
\cong H^{0,3}_!(X,\mathcal{V}_{\lambda})
\cong\bigoplus_{\substack{\pi_\infty\pi_{\mathrm{fin}}\in\mathcal{A}_0(GSp(4,\IA))\\\pi_\infty\cong\pi^H_{(k_1,k_2)}}} \pi_{\mathrm{fin}}\text{.}
\end{align} The sum is over the cuspidal automorphic irreducible representations $\pi$, whose archi\-me\-dean component is isomorphic to the holomorphic discrete series representation of weight $(k_1,k_2)=(\lambda_1+3,\lambda_2+3)$. For the non-holomorphic case we have
\begin{align}\label{eq:inner_cohom_XW}
 H^{2,1}_!(X,\mathcal{V}_{\lambda})\cong H^{1,2}_!(X,\mathcal{V}_{\lambda})
\cong\bigoplus_{\substack{\pi_\infty\pi_{\mathrm{fin}}\in\mathcal{A}_0(GSp(4,\IA))\\\pi_\infty\cong\pi^W_{(k_1,k_2)}}} \pi_{\mathrm{fin}}\text{,}
\end{align}
 the sum is over cuspidal automorphic irreducible representations $\pi$ with generic discrete series $\pi_{\infty}$ belonging to the same local $L$-packet as $\pi^H_{(k_1,k_2)}$.
For $\lambda_1>\lambda_2>0$ we have $H^i_!(X_K,\mathcal{V}_\lambda)=0$ for every $i\neq 3$ \citep{Taylor_Siegel_threefolds}, so this describes $H_{!}^\bullet(X,\mathcal{V}_{\lambda})$ completely.

There is a subspace $H_{!,E}^\bullet(X,\mathcal{V}_\lambda)$ of the inner cohomology $H^\bullet_!(X,\mathcal{V}_\lambda)$, which is maximal with respect to the property that its automorphic irreducible constituents are those representations of $GSp(4,\IA)$ which are weakly equivalent to constituents of globally parabolically induced representations. 
The archimedean factor of each constitutent of $H_{!,E}^{\bullet}(X,\mathcal{V}_\lambda)$ is in the discrete series and belongs to the Arthur packet attached to weight $(k_1,k_2)=(\lambda_1+3,\lambda_2+3)$. Each constituent is either a residue of an Eisenstein series or it is CAP. The residues of Eisenstein series only occur with Hodge type $(1,1)$ and $(2,2)$ \citep{Taylor_Siegel_threefolds}. The CAP-representations are those given by Thm.~\ref{thm:SK} and can occur with Hodge types $(1,1)$, $(2,2)$, $(3,0)$ and $(0,3)$. By \citep[Lemma 3]{Weissauer_trace_Hecke} we have
$H^\bullet_{!,E}(X,\mathcal{V}_\lambda)=0$ for $\lambda_1>\lambda_2>0$.

The orthocomplement of $H_{!,E}(X,\mathcal{V}_\lambda)$ with respect to the cup-product in $H_!(X,\mathcal{V}_\lambda)$ is composed of the cuspidal automorphic irreducible represen\-tations $\pi$ which are not CAP and whose archi\-me\-dean factor $\pi_{\infty}$ belongs to the local $L$-packet of the holomorphic discrete series representation of $GSp(4,\IR)$ with weight $(k_1,k_2)$. It splits into the direct sum $H_{!,\mathrm{endo}}^\bullet(X,\mathcal{V}_\lambda)\oplus H_{!,00}^\bullet(X,\mathcal{V}_\lambda)$.
The \emph{endoscopic part} $H_{!,\mathrm{endo}}^\bullet(X,\mathcal{V}_\lambda)$ is composed of cuspidal automorphic representations which are global weak endoscopic lifts, compare Def.~\ref{def:endolift}. It contributes with Hodge numbers $(3,0)$, $(2,1)$, $(1,2)$, $(0,3)$. The other part $H_{!,00}^\bullet(X,\mathcal{V}_\lambda)$ is composed of cuspidal irreducible representations that are neither CAP nor weak endoscopic lifts \citep{Weissauer_trace_Hecke}.

\subsection{Endoscopic Part}

\begin{Prop}\label{Prop:endo_lift_inncohom}
Suppose $\lambda_1\geq\lambda_2\geq0$ and let $r_1=\lambda_1+\lambda_2+4$ and $r_2=\lambda_1-\lambda_2+2$. Let $K\subseteq GSp(4,\hat{\IZ})$ be an open compact subgroup, then the endoscopic part of the inner cohomology of $X_K$ is given by
\begin{align}
\label{eq:endo_cohom1}
H_{!,\mathrm{endo}}^{(3,0)}(X_K,\mathcal{V}_\lambda)
&=\bigoplus_{\substack{\sigma_1,\sigma_2\in\mathcal{A}_0(GL(2,\IA))\\\sigma_{i,\infty}\cong\mathcal{D}(r_i-1)}}
\bigoplus_{\substack{S_{ng}
\\\#S_{ng}\text{ odd }}}\bigotimes_{v<\infty} \pi_v^{K(\mathfrak{p}_v)}\text{,}\\
\label{eq:endo_cohom2}
H_{!,\mathrm{endo}}^{(2,1)}(X_K,\mathcal{V}_\lambda)
&=\bigoplus_{\substack{\sigma_1,\sigma_2\in\mathcal{A}_0(GL(2,\IA))\\\sigma_{i,\infty}\cong\mathcal{D}(r_i-1)}}
\bigoplus_{\substack{S_{ng}\\\#S_{ng}\text{ even}}}\bigotimes_{v<\infty} \pi_v^{K(\mathfrak{p}_v)}\text{.}
\end{align}
The first sum runs over the cuspidal automorphic irreducible $GL(2,\IA)$-representations $\sigma_i$ with equal central character, whose archimedean factor is in the holomorphic discrete series of lowest weight $r_i$. The second sum is over the subsets $S_{ng}$ of the finite subsets of non-archimedean places $v$ where both $\sigma_{1,v}$ and $\sigma_{2,v}$ are in the discrete series. The factor $\pi_v$ is given by $\pi_v=\Pi_-(\sigma_v)$ for $v\in S_{ng}$ and $\pi_v=\Pi_+(\sigma_v)$ for $v\notin S_{ng}$.
\end{Prop}
\begin{proof}
 This is implied by Thm.~\ref{thm:Weissauer_main_thm} and Eqs.\ \eqref{eq:inner_cohom_XH}, \eqref{eq:inner_cohom_XW}.
\end{proof}

\begin{Corollary}\label{Cor:endo_lift_prime} Fix $\lambda_1\geq\lambda_2\geq0$ and let $K=K(\{p_0\})\subseteq GSp(4,\hat{\IZ})$ be the principal congruence subgroup of prime level $p_0$. Then the endoscopic part of the inner cohomology is
\begin{align*}
H_{!,\mathrm{endo}}^{(3,0)}(X_{K},\mathcal{V}_\lambda)
&\cong\bigoplus_{\sigma_1,\sigma_2}
\Pi_-(\sigma_{1,l},\sigma_{2,l})^{K((p_0))}\text{,}\\
H_{!,\mathrm{endo}}^{(2,1)}(X_{K},\mathcal{V}_\lambda)
&\cong\bigoplus_{\sigma_1,\sigma_2}
\Pi_+(\sigma_{1,l},\sigma_{2,l})^{K((p_0))}\text{.}
\end{align*}
The sum runs over the cuspidal automorphic irreducible representations $\sigma=(\sigma_1,\sigma_2)$ of $M(\IA)$ such that each archimedean $\sigma_{i,\infty}$ belongs to the discrete series of weight $r_i$ and that $\sigma_{i,v}$ is spherical at the non-archimedean places $v$ apart from  $v=p_0$.
\end{Corollary}
\begin{proof}
 For $\Pi_{+}(\sigma_v)$ in order to be spherical at every non-archimedean $v\neq p_0$ it is necessary and sufficient by Cor.~\ref{Cor:Local_K'Invariants} that the $\sigma_{i,v}$ are both spherical. 
The only possibilities for $S_{ng}$ in Prop.~\ref{Prop:endo_lift_inncohom} are $\{p_0\}$ for the first equation and the empty set for the second equation, respectively.
\end{proof}
Note that the sum is only over those automorphic representations $\sigma$ whose central character factors over $(\IZ/p_0\IZ)^\times$. For $p_0=2$ the central character must be trivial. Before we give the endoscopic contribution more explicitly, we first need to describe irreducible $GL(2,\IQ_2)$-representations of level $4$:
\begin{Lemma}\label{lem:level_4_rep_cuspidal}
Let $(\sigma,V_\sigma)$ be an irreducible representation of $GL(2,\IQ_2)$ of level $N=4$. Then $(\sigma,V_\sigma)$ is cuspidal, $\dim V_\sigma^{K^{(1)}(2)}=q-1$ for the first principal congruence subgroup $K^{(1)}(2)\subseteq GL(2,\IZ_2)$ and $\sigma$ is uniquely determined by its central character. Its local factors are $\epsilon(\sigma,s)=-1$ and $L(\sigma,s)=1$.
\end{Lemma}
\begin{proof}
By assumption the representation $(\sigma,V_\sigma)$ admits non-zero invariants under the subgroup $\{\left(\begin{smallmatrix}a&b\\c&d\end{smallmatrix}\right)\in GL(2,\IZ_2)\,|\,c\equiv0\mod 4\}$, which is conjugate to the principal congruence subgroup $K^{(1)}(2)\subseteq GL(2,\IZ_2)$. 
The action of $GL(2,\IZ_2)$ on the space of $K^{(1)}(2)$-invariants defines a non-zero representation $\tilde{\sigma}$ of $GL(2,\IZ_2)/K^{(1)}(2)\cong GL(2,\IF_2)$.
If $\tilde{\sigma}$ was non-cuspidal, then it would contain non-zero Borel-invariant vectors and so $\sigma$ would admit non-zero invariants under the standard Iwahori subgroup $K_0^{(1)}(2)\subseteq GL(2,\IZ_2)$. That is only possible for level $N\leq 2$, and so $\tilde{\sigma}$ must be cuspidal.
Now $\tilde{\sigma}$ is a sum of irreducible cuspidal representations of $GL(2,\IF_2)$, so by \citep[Thm.\ 2.23]{my_thesis} $\sigma$ itself must be cuspidal. Since $\sigma$ is irreducible, the same theorem implies that $\tilde{\sigma}$ must be irreducible.
Up to isomorphy there is only one cuspidal irreducible representation $\tilde\sigma$ of $GL(2,\IF_2)$ and by \citep[Thm.\ 2.23]{my_thesis} we have $\sigma\cong\cInd_{Z\,SL(2,\IZ_2)}^{GL(2,\IQ_2)}(\omega_\sigma\boxtimes\tilde{\sigma}|_{SL(2,\IZ_2)})$, so $\sigma$ is  uniquely determined by $\omega_\sigma$.

Fix an additive character $\psi:\IQ_2\to\IC^\times$ with $2\IZ_2\subseteq\ker(\psi)$, which gives rise to the non-trivial additive character $\tilde{\psi}:\IZ_2/2\IZ_2\to\IC^\times$. The $\epsilon$-factor of $\sigma$ is 
\begin{align*}
 \epsilon(\sigma,s,\psi)=2^{2\ell(\sigma)(\tfrac{1}{2}-s)}\frac{\tau(\Xi,\psi)}{(\mathfrak{A}:\mathfrak{P}^{n+1})^{1/2}}\text{,}
\end{align*}
in the notation of \citep[25.2, Theorem]{BushnellHenniart}. It is clear that $n=\ell(\sigma)=0$ since $\sigma^{K^{(1)}((2))}\neq0$ and that $(\mathfrak{A}:\mathfrak{P})=16$ because $\mathfrak{A}=\Mat_2(\IZ_2)$ and $\mathfrak{P}=I_2+\Mat_2(2\IZ_2)$. 
The cuspidal irreducible $\tilde\sigma$ corresponds to a non-trivial character $\Lambda$ of $\IF_4^\times$ with $\Lambda^2\neq\Lambda$. This $\Lambda$ is denoted by $\theta$ in \citep[\S 6.4]{BushnellHenniart}.
By (25.4.1) and the first equation in Section 23.7 in loc.\ cit.\ we have 
\begin{align*}
\tau(\Xi,\psi)
=-2\sum_{x\in\IF_4^\times}\overline{\Lambda(x)}\tilde{\psi}(x+x^2)=-4\text{.}
\end{align*}
The local $L$-factor is $L(\sigma,s)=1$ by \citep[\S24.5]{BushnellHenniart}.
\end{proof}

Let $\tau_{N,i}:=\dim \mathcal{S}_{r_i}(\Gamma_0(N))^{\text{new}}$ denote the dimension of the space of newforms of weight $r_i$ and level $N$. Furthermore, let $\tau^\pm_{i}$ denote the dimension of the subspace of level $N=2$ newforms of weight $r_i$ and with Atkin-Lehner eigenvalue $\pm1$. Recall that the principal congruence subgroup $K=K(\left\{2\right\})$ of level $2$ is a normal subgroup of $GSp(4,\hat{\IZ})$ that satisfies the conditions for strong approximation. We use the notation of Enomoto \citep{Enomoto1972}.
\begin{Corollary}\label{Cor:BFG_endo_conjs}
Suppose $K\subseteq GSp(4,\hat{\IZ})$ is the principal congruence subgroup of level $2$ and let $\lambda_1\geq\lambda_2\geq0$. Then the endoscopic part of the inner cohomology is given by
\begin{align*}
 H_{!,\mathrm{endo}}^{(3,0)}(X_{K},\mathcal{V}_\lambda)
=&(\tau_{1}^+\tau_{2}^-+\tau_{1}^-\tau_{2}^+)\cdot\theta_5
+(\tau_{1}^+\tau_{2}^++\tau_{1}^-\tau_{2}^-)\cdot\theta_2
+\tau_{4,1}\tau_{4,2}\cdot\chi_9(1)\text{ and}\\
\ \\
H_{!,\mathrm{endo}}^{(2,1)}(X_{K},\mathcal{V}_\lambda)
=&\tau_{4,1}\tau_{4,2}\cdot\chi_{13}(1)+(\tau_{4,1}\tau_{2,2}+\tau_{2,1}\tau_{4,2})\cdot\chi_{12}(1)\\
 &+(\tau_{1}^+\tau_{2}^++\tau_{1}^-\tau_{2}^-)\cdot(\theta_1+\theta_4)+(\tau_{1}^+\tau_{2}^-+\tau_{1}^-\tau_{2}^+)\cdot(\theta_3+\theta_4)\\
 &+(\tau_{1,1}\tau_{4,2}+\tau_{4,1}\tau_{1,2})\cdot(\chi_8(1)+\chi_{12}(1))\\
 &+(\tau_{1,1}\tau_{2,2}+\tau_{2,1}\tau_{1,2})\cdot(\theta_1+\theta_3+\theta_4)\\
 &+\tau_{1,1}\tau_{1,2}\cdot(\theta_0+2\theta_1+\theta_2+\theta_3+\theta_4)\text{.}
\end{align*}
\end{Corollary}

\begin{proof} We begin with the first equation.
By Corollary~\ref{Cor:endo_lift_prime} for each given $\sigma=(\sigma_1,\sigma_2)$ we consider exactly one holomorphic lift with non-archimedean part
\begin{align*}
\pi_{\mathrm{fin}}=\Pi_-(\sigma_{1,2},\sigma_{2,2})\otimes\bigotimes_{2<v<\infty}\Pi_{+}(\sigma_{1,v},\sigma_{2,v})\text{.}
\end{align*} and thus $\pi_{\mathrm{fin}}^K\cong (\Pi_{-}(\sigma_{1,2},\sigma_{2,2}))^{K_2}$. Therefore by Prop.~\ref{Prop:Local_KInvariants} $\sigma_{1,2}$ and $\sigma_{2,2}$ must both admit non-zero invariants under the principal congruence subgroup $K^{(1)}$ in $GL(2,\hat{\IZ})$ of principle congruence subgroup level two, else $\pi_{\mathrm{fin}}^K$ would be zero.
The cuspidal $\sigma_{i,2}$ with non-zero $K^{(1)}(2\IZ_2)$-invariants are exactly those with level $N=4$ by Lemma~\ref{lem:level_4_rep_cuspidal}, so the number of relevant automorphic irreducible representations $\sigma_i$ with cuspidal $\sigma_{i,2}$ is exactly $\tau_{4,i}$. The number of $\sigma_i$ with $\sigma_{i,2}\cong\St_{GL(2)}$ and $\sigma_i^{K^{(1)}(2\IZ_2)}\neq0$ is $\tau_{i}^-$ while the corresponding number for $\sigma_{i,2}\cong\xi_u\cdot\St_{GL(2)}$ is $\tau_{i}^+$, where $\xi_u$ is the unramified non-trivial quadratic character of $\IQ_2^\times$ \citep[Prop.\ 5.21]{Gelbart}.
Now Proposition~\ref{Prop:Local_KInvariants} describes the invariant spaces $\pi^{K}$ as representations of $GSp(4,\IF_2)$. 

The proof for the second equation is analogous. We use the following decompositions into irreducibles: $\chi_1(0,0)=\theta_0+2\theta_1+\theta_2+\theta_3+\theta_4$ and $\chi_2(1)=\chi_8(1)+\chi_{12}(1)$ and $\chi_{10}(0)=\theta_1+\theta_3+\theta_4$.
\end{proof}

For $K=K(\{2\})$ and $\lambda_1=\lambda_2$ we have $H_{!,\mathrm{endo}}^{3}(X_{K},\mathcal{V}_\lambda)=0$, since
$\mathcal{S}_2(\Gamma_0(2))=0$.

\begin{Corollary}\label{Cor:strict_endoscopy} Suppose $\lambda_1\geq\lambda_2\geq0$ and $K\subseteq GSp(4,\hat{\IZ})$ is the principal congruence subgroup of level two. 
Then 
\begin{align}\label{eq:diff_H30H21}
\dim H^{(2,1)}_{!,\mathrm{endo}}(X_K,\mathcal{V}_\lambda)-\dim H^{(3,0)}_{!,\mathrm{endo}}(X_K,\mathcal{V}_\lambda)
=5\cdot\dim \mathcal{S}_{r_1}(\Gamma_0(4))\cdot\dim \mathcal{S}_{r_2}(\Gamma_0(4))\text{.}
\end{align}
\end{Corollary}
\begin{proof}
Table~\ref{tab:invar_local_endo} implies for generic irreducible representations $\sigma_{i,2}$ of $GL(2,\IQ_2)$:
\begin{align*}
\dim\Pi_+(\sigma_{1,2},\sigma_{2,2})^{K_2}-\dim\Pi_-(\sigma_{1,2},\sigma_{2,2})^{K_2}=5\cdot\dim\sigma_{1,2}^{K^{(1)}(2\IZ_2)}\cdot\dim\sigma_{2,2}^{K^{(1)}(2\IZ_2)}\text{,}
\end{align*}
 compare \citep[Cor.\ 4.12]{my_thesis}. The dimension of $\sigma_{i,2}^{K^{(1)}(2\IZ_2)}$ is given by \citep[Table 2.1]{my_thesis}. The right hand side of \eqref{eq:diff_H30H21} equals 
$5\cdot (\tau_{1,1}\cdot3+\tau_{2,1}\cdot 2+\tau_{4,1}\cdot 1)\cdot(\tau_{1,2}\cdot 3+\tau_{2,2}\cdot 2+\tau_{4,2}\cdot1)$ by Atkin-Lehner theory.
\end{proof}
Note that by \citep[Thm.\ 1]{Weissauer_trace_Hecke} the subspaces $H^{(3,0)}_{!,00}(X_K,\mathcal{V}_{\lambda})$ and $H^{(2,1)}_{!,00}(X_K,\mathcal{V}_{\lambda})$ are iso\-morphic. For $\lambda_1>\lambda_2>0$ we have $H_{!,E}(X_K,\mathcal{V}_\lambda)=0$, so in this case:
\begin{align}
\dim H^{(2,1)}_{!}(X_K,\mathcal{V}_\lambda)-\dim H^{(3,0)}_{!}(X_K,\mathcal{V}_\lambda)
=5\cdot\dim \mathcal{S}_{r_1}(\Gamma_0(4))\cdot\dim \mathcal{S}_{r_2}(\Gamma_0(4))\text{.}
\end{align}

\subsection{Saito-Kurokawa Part}
 We now consider the subspace $H^\bullet_{!,\mathrm{SK}}(X,\mathcal{V}_\lambda)\subseteq H^\bullet_{!,\mathrm{E}}(X,\mathcal{V}_\lambda)$ that is composed of automorphic irreducible representations with trivial central character coming from Saito-Kurokawa Lifts in the sense of Piatetski-Shapiro, compare Thm.~\ref{thm:SK}.
 These lifts are all weakly equivalent to globally parabolically induced representations from the Borel or Siegel parabolic \citep[Thm.\ 2.2]{PS-SaitoK}, so they are either CAP or not cuspidal. We have already mentioned that $H_E^\bullet(X_K,\mathcal{V}_\lambda)=0$ for $\lambda_1>\lambda_2>0$, so we assume $\lambda_1=\lambda_2$.

\begin{Prop}Suppose $\lambda_1=\lambda_2\geq0$, let $k=\lambda_1+3$ and $r=2k-2$. For an arbitrary open compact subgroup $K\subseteq GSp(4,\hat{\IZ})$, the Saito-Kurokawa part of the inner cohomology of $X_K$ is given by
\begin{align}
\label{eq:SK_cohom1}
H_{!,\mathrm{SK}}^{(3,0)}(X_K,\mathcal{V}_\lambda)\cong H_{!,\mathrm{SK}}^{(0,3)}(X_K,\mathcal{V}_\lambda)
&\cong\bigoplus_{\substack{\sigma\in\mathcal{A}_0(PGL(2,\IA))\\\sigma_{\infty}\cong\mathcal{D}(r-1)}}
\bigoplus_{
 S\ni\infty}\;\;\bigotimes_{v<\infty}\Pi(\sigma_v,\sigma_{S,v})^{K(\mathfrak{p}_v)}\text{,}\\
H_{!,\mathrm{SK}}^{(1,1)}(X_K,\mathcal{V}_\lambda)\cong H_{!,\mathrm{SK}}^{(2,2)}(X_K,\mathcal{V}_\lambda)
&\cong\bigoplus_{\substack{\sigma\in\mathcal{A}_0(PGL(2,\IA))\\\sigma_{\infty}\cong\mathcal{D}(r-1)}}
\bigoplus_{S\not\ni \infty}\;\; \bigotimes_{v<\infty}\Pi(\sigma_v,\sigma_{S,v})^{K(\mathfrak{p}_v)}\text{.}
\label{eq:SK_cohom2}
%
\end{align}
The sum is over the finite subsets $S$ of non-archimedean places where $\sigma_v$ is in the discrete series satisfying the condition $(-1)^{\#S}=\epsilon(\sigma,1/2)$.
\end{Prop}
\begin{proof} By construction, the sum is over the Saito-Kurokawa lifts given by Thm.~\ref{thm:SK} with archimedean factor in the holomorphic discrete series. Those with Hodge-type $(3,0)$ and $(0,3)$ are the ones with holomorphic archimedean factor, hence we need to assume $\infty\in S$. The Hodge-types $(1,1)$ and $(2,2)$ belong to the $\sigma$ with non-tempered archimedean factor, i.e.\ $\infty\notin S$.
\end{proof}

We now look at the special case where $K\subseteq GSp(4,\hat{\IZ})$ is the principal congruence subgroup of level $2$. As before, let $\tau_{N,r}:=\dim \mathcal{S}_{r}(\Gamma_0(N))$ and let $\tau_r^\pm$ denote the dimenion of the subspace of $\mathcal{S}_r(\Gamma_0(2))$ with Atkin-Lehner eigenvalue $\pm1$. 

\begin{Corollary}\label{Cor:CAP_cohom_lev2} For $\lambda_1=\lambda_2\geq0$ let $k=\lambda_1+3$ and $r=2k-2$. For the principal congruence subgroup $K\subseteq GSp(4,\hat{\IZ})$ of level $2$, the Saito-Kurokawa-part of the inner cohomology of $X_K$ is given by
\begin{align*}
H_{!,\mathrm{SK}}^{(3,0)}(X_K,\mathcal{V}_\lambda)&\cong H_{!,\mathrm{SK}}^{(0,3)}(X_K,\mathcal{V}_\lambda)\\
&=\begin{cases}
\tau_r^+\cdot \theta_1+\tau_r^-\cdot \theta_2+\tau_{1,r}\cdot (\theta_0+\theta_1+\theta_2) & \text{ for }k\text{ even,}\\
\tau_{4,r}\cdot \chi_8(1)+\tau_r^+\cdot\theta_5+\tau_r^-\cdot \theta_3 & \text{ for } k\text{ odd,}
\end{cases}\\ 
\ \\
H_{!,\mathrm{SK}}^{(1,1)}(X_K,\mathcal{V}_\lambda)&\cong H_{!,\mathrm{SK}}^{(2,2)}(X_K,\mathcal{V}_\lambda)\\
&=\begin{cases}
\tau_{4,r}\cdot \chi_8(1)+\tau_r^+\cdot\theta_5+\tau_r^-\cdot \theta_3 & \text{ for } k\text{ even,}\\
\tau_r^+\cdot \theta_1+\tau_r^-\cdot \theta_2+\tau_{1,r}\cdot (\theta_0+\theta_1+\theta_2) & \text{ for }k\text{ odd.}\\
\end{cases}
\end{align*}
\end{Corollary}

\begin{proof}
 We begin with $H_{!,\mathrm{SK}}^{(3,0)}(X_K,\mathcal{V}_\lambda)$.
Suppose some automorphic Saito-Kurokawa lift $\pi$ gives a non-zero contribution to this space, then the central character $\omega_\pi$ must factor over $(\IZ/2\IZ)^\times$, hence $\omega_\sigma=\omega_\pi=1$.
We need $\pi_\infty$ to be isomorphic to the holomorphic discrete series representation of weight $(k,k)$ so we only consider those $\sigma_k$ in the discrete series of $PGL(2,\IR)$ with lowest weight $k$. At the non-archimedean places $v\neq2$ the representation $\pi_v$ must be spherical and by Prop.~\ref{prop:SK_invars} $\sigma_v$ must then also be spherical.
By Prop.~\ref{prop:SK_invars} the representation $\pi_2$ can admit non-zero invariants under $K_2=K(2\IZ_2)$ only if $\sigma_2$ admits non-zero $K^{(1)}(2\IZ_2)$-invariants.

The non-cuspidal possibilities for $\sigma_2$ are $\sigma_2=\St_{GL(2,\IQ_2)}$ and $\sigma_2=\xi\cdot\St_{GL(2,\IQ_2)}$ and the principal series representation $\sigma_2=\mu\times\mu^{-1}$, where $\mu:\IQ_2^\times\to\IC^\times$ is a smooth character. Non-zero $K^{(1)}(2\IZ_2)$-invariant vectors in $\sigma_2$ only occur for at most tamely ramified $\mu$ and $\xi$ and for $\IQ_2^\times$ that already implies unramified.
For even $k$ condition \eqref{eq:SK_cond} permits the lifts $\pi_2=\Pi(\St_{GL(2,\IQ_2)},\St_{GL(2,\IQ_2)})$ and $\pi_2=\Pi(\xi_u\cdot\St_{GL(2,\IQ_2)},\Ione_{GL(2,\IQ_2)})$ and $\pi_2=\Pi(\mu\times\mu^{-1},\Ione_{GL(2,\IQ_2)})$, which occur with multiplicity $\tau_r^-$, $\tau_r^+$ and $\tau_{1,r}$, respectively. For odd $k$ by condition \eqref{eq:SK_cond} there are the lifts $\pi_2=\Pi(\St_{GL(2,\IQ_2)},\Ione_{GL(2,\IQ_2)})$ and $\pi_2=\Pi(\xi_u\cdot\St_{GL(2,\IQ_2)},\St_{GL(2,\IQ_2)})$ occuring with multiplicities $\tau_r^-$ and $\tau_r^+$. Prop.~\ref{prop:SK_invars} describes the invariant space $\pi_2^{K_2}\cong\mathcal{F}_{\mathfrak{p}_2}(\pi_2)$ as a representation of $GSp(4,\IZ_2/2\IZ_2)\cong Sp(4,\IF_2)$; note that $\chi_6(0)=\theta_0+\theta_1+\theta_2$.

Now consider the case where $\sigma_2$ is cuspidal and $\sigma_2^{K^{(1)}(2\IZ_2)}\neq0$. The central $\epsilon$-value is $\epsilon(\sigma,1/2)=(-1)^k$ by Lemma~\ref{lem:level_4_rep_cuspidal}. For even $k$ there is a Saito-Kurokawa-lift $\pi$ with $\pi_2=\Pi(\sigma_2,\St_{GL(2,\IQ_2)})$. By Prop~\ref{prop:SK_invars} this lift does not admit non-zero $K_2$-invariants. For odd $k$ we have the lift $\pi_2=\Pi(\sigma_2,\Ione_{GL(2,\IQ_2)})$ and the invariant space is isomorphic to $\chi_8(1)$ as a representation of $Sp(4,\IZ_2/2\IZ_2)$.

The proof for the second equation is analogous. Note that the lifts are not necessarily cuspidal anymore.
\end{proof}

\begin{Corollary}\label{Cor:strict_endoscopy_SK} Suppose $\lambda_1=\lambda_2\geq0$ and let $r=2k-2$. For the principal congruence subgroup $K\subseteq GSp(4,\hat{\IZ})$ of level $2$ we have the equation
\begin{align}\label{eq:diff_H30H21SK}
\dim H^{(1,1)}_{!,\mathrm{SK}}(X_K,\mathcal{V}_\lambda)+\dim H^{(3,0)}_{!,\mathrm{SK}}(X_K,\mathcal{V}_\lambda)
=5\cdot\dim \mathcal{S}_{r}(\Gamma_0(4))\text{.}
\end{align}
\end{Corollary}
\begin{proof}
The left hand side is given by Corollary \ref{Cor:CAP_cohom_lev2} and Table \ref{tab:dic}. The proof is analogous to Corollary~\ref{Cor:strict_endoscopy}.
\end{proof}

\appendix
\begin{small}
\bibliographystyle{plainnat}
\bibliography{L-factors}
\end{small}
\end{document}